\documentclass[12pt]{amsart}
\usepackage[latin1]{inputenc}
\usepackage[english]{babel}
\usepackage{amsmath,amssymb,amsthm,amsfonts, color}
\usepackage{enumerate}
\usepackage{graphicx}
\usepackage{latexsym}
\usepackage[all]{xy}
\usepackage{tikz}

\makeatletter
\@namedef{subjclassname@2010}{ \textup{2010} Mathematics Subject Classification}
\makeatother

\newtheorem{lemma}{Lemma}
\newtheorem{remark}{Remark}
\theoremstyle{definition}
\newtheorem{defin}{Definition}
\newtheorem{prop}[lemma]{Proposition}
\newtheorem{cor}[lemma]{Corollary}
\newtheorem{theor}{Theorem}

\theoremstyle{definition}

\begin{document}
\title{Symmetric finite representability of $\ell^p$-spaces in rearrangement invariant spaces on $(0,\infty)$.}
\thanks{{\rm *}The work was completed as a part of the implementation of the development program of the Scientific and Educational Mathematical Center Volga Federal District, agreement no. 075-02-2020-1488/1.}
\author[Astashkin]{Sergey V. Astashkin}
\address[Sergey V. Astashkin]{Department of Mathematics, Samara National Research University, Moskovskoye shosse 34, 443086, Samara, Russia
}
\email{\texttt{astash56@mail.ru}}
\maketitle

\vspace{-7mm}

\begin{abstract}
For a separable rearrangement invariant space $X$ on $(0,\infty)$ of fundamental type we identify the set of all $p\in [1,\infty]$ such that $\ell^p$ is finitely represented in $X$ in such a way that the unit basis vectors of $\ell^p$ ($c_0$ if $p=\infty$) correspond to pairwise disjoint and equimeasurable functions. This characterization hinges upon a description of the set of approximate eigenvalues of the doubling operator $x(t)\mapsto x(t/2)$ in $X$. We prove that this set is surprisingly simple: depending on the values of some dilation indices of such a space, it is either an interval or a union of two intervals. We apply these results to the Lorentz and Orlicz spaces.  
\end{abstract}


\newcommand{\fg}{\mathbb N}
\newcommand{\gh}{\mathbb R}
\newcommand{\hj}{\mathbb Z}
\def\mn{\tau_{\lambda}}
\def\op{{\rm Im}\,\tau_{\lambda}}
\def\zw{{\rm Ker}\,f_\lambda}

\section{Introduction.}
\label{Intro}

In 1974, Tsirelson \cite{Tsi} constructed an example of a Banach space containing no isomorphic copies of $\ell^p$, $1\le p<\infty,$ and $c_0$. Two years later, Krivine proved a theorem showing an essential difference between {\it global} and {\it local} properties of Banach spaces, that is, between properties of their infinite-dimensional subspaces  and subspaces of finite (though large) dimension. To state and discuss this result, we define some notions.  

\begin{defin}
\label{def1}
Suppose $1\le p\le\infty$, $X$ is a Banach space,  and $\{z_i\}_{i=1}^\infty$ is a bounded sequence in $X$. The space $\ell^p$ is said to be {\it block finitely represented in $\{z_i\}_{i=1}^\infty$} if for every
$n\in\mathbb{N}$ and $\varepsilon>0$ there exist $0=m_0<m_1<\dots <m_n$ and $\alpha_i\in\mathbb{R}$ such that the vectors $u_k=\sum_{i=m_{k-1}+1}^{m_k}\alpha_iz_i$, $k=1,2,\dots,n$, satisfy the inequality
\begin{equation}
\label{eq: blocks}
(1+\varepsilon)^{-1}\|a\|_p\le \Big\|\sum_{k=1}^n
a_ku_k\Big\|_X\le (1+\varepsilon)\|a\|_p
\end{equation}
for arbitrary $a=(a_k)_{k=1}^n\in\mathbb{R}^n$. In what follows, as usual,
$$
\|a\|_p:=\Big(\sum_{k=1}^n
|a_k|^p\Big)^{1/p} \enskip\text{if}\enskip 
p<\infty,\enskip\text{and}\enskip
\|a\|_\infty:=\sup_{k=1,2,\dots,n}|a_k|
$$
(including the case when $n=\infty$).
\end{defin}

\begin{defin}
Let $X$ be a Banach space, and let $1\le p\le\infty$. The space $\ell^p$ is said to be {\it finitely
represented} in $X$ if for every $n\in\mathbb{N}$ and $\varepsilon>0$ there exist
$x_1,x_2,\dots,x_n\in X$ such that for every $a=(a_k)_{k=1}^n$
\begin{equation}\label{eq: 1}
(1+\varepsilon)^{-1}\|a\|_p\le \Big\|\sum_{k=1}^n
a_kx_k\Big\|_X\le (1+\varepsilon)\|a\|_p.
\end{equation}
\end{defin}

Long before Tsirelson's example, in 1961, Dvoretzky proved his  celebrated theorem (see \cite{Dvor}, \cite[Theorem 5.8]{MilSch} or \cite[Theorem 11.3.13]{AK}), showing that $\ell^2$ is finitely represented in an arbitrary infinite-dimensional Banach space $X$. Clearly, if $\ell^p$ is block finitely represented in some sequence $\{z_i\}_{i=1}^\infty\subset X$, then $\ell^p$ is finitely represented in $X$. Hence, the following famous result proved by Krivine in \cite{Kriv} (see also \cite{Ros} and \cite[Theorem 11.3.9]{AK}) became a very valuable supplement to the Dvoretzky theorem.

\begin{theor}\label{Th:Krivine1} 
Let $\{z_i\}_{i=1}^\infty$ be an arbitrary normalized sequence in a Banach space $X$ such that the vectors $z_i$ do not form a relatively compact set. Then $\ell^p$ is block finitely
represented in $\{z_i\}_{i=1}^\infty$ for some $p$, $1\leq p\leq 
\infty$.
\end{theor}

Here, we consider a special class of Banach lattices, namely,  rearrangement invariant (in other terminology, symmetric) function spaces on $(0,\infty)$ (for all definitions see the next section). The main aim of this paper is to find a description of the set of all $p$ such that $\ell^p$ is finitely represented in such a space so that the unit basis vectors of $\ell^p$ correspond to equimeasurable and pairwise disjoint functions.

Measurable functions $x(t)$ and $y(t)$ on $(0,\infty)$ are called {\it equimeasurable} if
$$
m\{s>0:\,|x(s)|>\tau\}=m\{s>0:\,|y(s)|>\tau\}\;\;\mbox{for all}\;\;\tau>0.$$
Here and in the sequel, $m$ denotes the Lebesgue measure. 

\begin{defin}
Let $X$ be a rearrangement invariant space on $(0,\infty)$, $1\le p\le\infty$. We say that $\ell^p$ is {\it symmetrically finitely represented} in $X$ if for every $n\in\fg$ and each $\varepsilon>0$ there exist equimeasurable functions $x_k\in X$, $k=1,2,\dots,n$, such that ${\rm supp}\,x_i\cap {\rm supp}\,x_j=\varnothing$, $i\ne j$, and for any $a=(a_k)_{k=1}^n$ 
\begin{equation}
\label{eq0} (1+\varepsilon)^{-1}\|a\|_p\le \Big\|
\sum_{k=1}^na_kx_k\Big\|_X\le (1+\varepsilon) \|a\|_p.
\end{equation}
\end{defin}

Let $\{z_i\}_{i=1}^\infty$ be a sequence in a rearrangement invariant space $X$ on $(0,\infty)$. If for all $n$ and $\epsilon>0$ we can find  blocks $u_k=\sum_{i=m_{k-1}+1}^{m_k}\alpha_iz_i$, $k=1,2,\dots,n$, as in Definition \ref{def1}, which are being pairwise disjoint,  equimeasurable functions and satisfy \eqref{eq: blocks}, we say that $\ell^p$ is {\it symmetrically block finitely represented} in the sequence $\{z_i\}$.

We introduce also a weaker notion.

\begin{defin}
The space $\ell^p$ is {\it crudely symmetrically finitely represented} in a rearrangement invariant function space $X$ if there exists a constant $C>0$ such that for every $n\in\fg$ we can find 
equimeasurable functions $x_k\in X$, $k=1,2,\dots,n$, such that ${\rm supp}\,x_i\cap {\rm supp}\,x_j=\varnothing$, $i\ne j$, and for every $a=(a_k)_{k=1}^n$ 
\begin{equation*}
\label{eq0} C^{-1}\|a\|_p\le \Big\|
\sum_{k=1}^na_kx_k\Big\|_X\le C\|a\|_p.
\end{equation*}
\end{defin}

The set of all $p\in [1,\infty]$ such that  $\ell^p$ is symmetrically finitely represented (resp. crudely symmetrically finitely represented) in $X$ we will denote by ${\mathcal F}(X)$ (resp. ${\mathcal F}_c(X)$).
Observe that knowledge of the structure of the above sets may be rather useful when studying various properties of rearrangement invariant spaces and  operators bounded in these spaces (see e.g. \cite{A16}, \cite{AMT-13}). 

In 1978,  Rosenthal \cite{Ros} proved the following version of Theorem \ref{Th:Krivine1}.

\begin{theor}\label{Th:Rosenthal} 
Let $\{e_i\}_{i=1}^\infty$ be a subsymmetric unconditional basis for a Banach space $X$ such that $c_0$ is not finitely
represented in $X$. If the numbers $s_n$ are defined by $\|\sum_{i=1}^ne_i\|_X=n^{1/s_n}$, then $\ell^p$ is block finitely represented in $\{e_i\}$ provided that  
$$
p\in[\liminf_{n\to\infty}s_n,\limsup_{n\to\infty}s_n].$$ 
\end{theor}
Moreover, as it follows from an inspection of the proof of this theorem (see also the remark after the proof of Theorem 3.3 in \cite{Ros}), for each $p$ satisfying the hypothesis, any $n$ and $\epsilon>0$ there exist blocks $u_k$, $k=1,2,\dots,n$, of the basis $\{e_i\}$, with the same ordered distribution, such that we have \eqref{eq: blocks} (blocks $u=\sum_{i=1}^{m}\alpha_ie_{n_i}$ and $v=\sum_{i=1}^{m}\beta_ie_{k_i}$ have the same ordered distribution if $\beta_i=\alpha_{\pi(i)}$,  $i=1,2,\dots,m$, for some permutation $\pi$ of the set $\{1,2,\dots,m\}$; see \cite{Ros}). 

Theorem \ref{Th:Rosenthal} suggests that a condition ensuring that $\ell^p$ is block finitely represented in a subsymmetric unconditional (in particular, symmetric) basis $\{e_i\}$ of a Banach space so that the unit basis vectors of $\ell^p$ correspond to blocks of this basis  with the same ordered distribution can be expressed by using some estimates for the norms of an appropriate dilation operator when it is restricted to the set $\{e_i\}$. Clearly, any rearrangement invariant space on $(0,\infty)$ fails to have such a basis. However, as we show here, a careful studying properties of the doubling operator $\sigma x(t):=x(t/2)$, allows us to establish a complete description of the set ${\mathcal F}(X)$ for a wider class of separable rearrangement invariant spaces on $(0,\infty)$. 

Next, a key role will be played by the following notion.

\begin{defin}
Let  
$T$ be a bounded linear operator on a Banach space $X$. A sequence
$\{u_n\}_{n=1}^\infty\subset
X$, $\|u_n\|=1$, $n=1,2,\dots$, is called an {\it approximate eigenvector} corresponding to an {\it approximate eigenvalue} $\lambda\in\gh$ for $T$ if
$\|Tu_n-\lambda u_n\|_X\to 0$ \footnote{In what follows, we deal only with real Banach spaces. Note however that every bounded linear  operator in a Banach spaces over the complex field has at least one approximate eigenvalue (see \cite[12.1]{MilSch})}. 
\end{defin}

We will make use of the following characterization of the set ${\mathcal F}(X)$ for a separable rearrangement invariant space on $(0,\infty)$ obtained in the paper \cite{A-11}.

\begin{theor}
\label{Theorem 1}
Let $X$ be an arbitrary separable rearrangement invariant space on $(0,\infty)$. The following conditions are equivalent:

(a) $p\in{\mathcal F}(X)$;

(b) $p\in{\mathcal F}_c(X)$;

 (c) $2^{1/p}$ is an approximate eigenvalue of the doubling operator $\sigma$\footnote{In \cite{A-11}, it is proved that $(a)\Longleftrightarrow (c)$; the equivalence of these conditions to (b) follows immediately from the proofs given there.}.
\end{theor}

Thus, to characterize the set ${\mathcal F}(X)$, it suffices to identify the set of approximate eigenvalues of the operator $\sigma$ in $X$. It can be rather easily shown (see e.g. \cite[Lemma~11.3.12]{AK} or  \cite[Theorem~6]{A-11}) that the latter set is contained in the interval $[2^{\alpha_X},2^{\beta_X}]$, where $\alpha_X$ and $\beta_X$ are the Boyd indices of $X$. We are able to get a complete description of this set (which may not coincide with the above interval) for rearrangement invariant spaces of fundamental type. The latter property means, roughly speaking, that each of the norms of the dilation operators $x(t)\mapsto x(2^nt)$, $n\in\hj$, in $X$ can be calculated (up to equivalence with a constant independent of $n$) when these operators are restricted to the set of all characteristic functions. It should be emphasized that this condition is not so restrictive; all the  most known and important rearrangement invariant spaces, e.g., Orlicz, Lorentz, Marcinkiewicz spaces, are of fundamental type.

The main result expressed in terms of appropriate dilation indices (see their definition in the next section) can be stated as follows.

\begin{theor}\label{Theorem 4}
Let $X$ be a rearrangement invariant space on $(0,\infty)$ of fundamental type.  

(i) If ${\alpha_X^\infty}\le {{\beta_X^0}}$, then the set of approximate eigenvalues of the operator $\sigma$ in $X$ is the interval $[2^{\alpha_X},2^{\beta_X}]$.

(ii) If ${\alpha_X^\infty}>{{\beta_X^0}}$, then the set of approximate eigenvalues of $\sigma$ in $X$ is the union  $[2^{\alpha_X},2^{\beta_X^0}] \cup [2^{\alpha_X^\infty},2^{\beta_X}]$.
\end{theor}

From Theorems \ref{Theorem 4} and \ref{Theorem 1} we get the following description of the set of $p\in [1,\infty]$ such that  $\ell^p$ is symmetrically finitely represented in a given separable rearrangement invariant space on $(0,\infty)$ of fundamental type.

\begin{cor}\label{main1}
Let $X$ be a separable rearrangement invariant space on $(0,\infty)$ of fundamental type. Then, we have 

(i) if ${\alpha_X^\infty}\le {{\beta_X^0}}$, then ${\mathcal F}(X)={\mathcal F}_c(X)=[1/\beta_X,1/\alpha_X]$;

(ii) if ${\alpha_X^\infty}>{{\beta_X^0}}$, then ${\mathcal F}(X)={\mathcal F}_c(X)=[1/\beta_X,1/\alpha_X^\infty]\cup [1/\beta_X^0,1/\alpha_X]$.
\end{cor}

Our approach to the problem of finding the set of approximate eigenvalues of the doubling operator $\sigma$ in a rearrangement invariant space $X$ is based on its reduction to the similar task for the shift operator $\tau (a_k)=(a_{k-1})$ in some Banach sequence lattice $E_X$ modelled on $\hj$ such that the sequence of characteristic functions of dyadic intervals $\{\chi_{\Delta_k}\}_{k\in\hj}$, where $\Delta_k=[2^k,2^{k+1})$, is equivalent in $X$ to the unit vector basis in $E_X$ (see Proposition \ref{Proposition 1}). 
This is a development of an idea applied by Kalton \cite{Kal} in the special case of Lorentz spaces and corresponding weighted $\ell^p$-spaces when studying completely different problems related to interpolation theory of operators, in particular, to a characterization of Calder\'{o}n couples of r.i. spaces (see also \cite[Lemma~2.2]{IK-01}).  

A description of the set of all positive $\lambda$, for which the operator $\tau_\lambda=\tau-\lambda I$ is an isomorphism in a Banach sequence lattice $E$ satisfying certain properties (and thereby of the set of approximate eigenvalues of $\tau$), is given in Theorem \ref{L2-new}. In Appendix (see Theorem \ref{L2-new-close}), we put a characterization of the set of positive $\lambda$ such that the operator $\tau_\lambda$ is closed in such a  lattice $E$, which complements Theorem \ref{L2-new}.

Comparing Theorems  \ref{Theorem 4} and \ref{Th:Rosenthal} and  their proofs indicates that, when one considers the problem of symmetric finite representability of $\ell^p$-spaces in rearrangement invariant spaces, then the sequence of characteristic functions of dyadic intervals $\Delta_n^k=[2^{-n}(k-1), 2^{-n}k)$, $n=0,1,\dots$, $k=1,2,\dots$, has properties similar to those of a subsymmetric unconditional basis in a space with such a basis. In particular, our approach allows to obtain for the sequence $\{\chi_{\Delta_{n}^k}\}_{n,k}$ and {\it general} separable rearrangement invariant spaces a result analogous to Theorem \ref{Th:Rosenthal} (see Theorem \ref{Theorem 5f}). 

In the concluding part of the paper, we identify the set ${\mathcal F}(X)$ for arbitrary Lorentz and separable Orlicz spaces (Theorems \ref{Theorem 4a} and \ref{Theorem 4b}).  

\vskip0.5cm

\section{Preliminaries.}
\subsection{Banach function and sequence lattices.}
\label{prel1}
Here, we recall some definitions and results that relate to Banach function lattices; for a detailed exposition see, for example, the monographs \cite{KA,BSh,LT1,LT2}.

Let $(\Omega,\Sigma,\mu)$ be a $\sigma$-finite measure space and let $S(\Omega,\Sigma,\mu)$ be the linear topological space of all a.e. finite real-valued functions (of equivalence classes) defined on $\Omega$ with natural algebraic operations and the topology of convergence in measure $\mu$ on sets of finite measure.

A Banach space $E \subset S(\Omega,\Sigma,\mu)$ is said to be a Banach {\it function lattice} (or {\it ideal space}) if from $x\in E$, $y \in S(\Omega,\Sigma,\mu)$ and $|y| \leq |x|$ a.e. it follows that $y\in E$ and $\|y\|_E \leq \|x\|_E$. In the case when $\Omega=\mathbb{Z}$ and 
$\mu$ is the counting measure, we will say that $E$ is a Banach sequence lattice.

Every Banach function lattice $E$ over a $\sigma$-finite measure space $(\Omega,\Sigma,\mu)$ is linearly and continuously embedded into $S(\Omega,\Sigma,\mu)$ (see e.g. \cite[Theorem~4.3.1]{KA}). This means that convergence in the $E$-norm implies convergence in measure on all subsets of $\Omega$ of finite measure.

If $E$ is a Banach function lattice, then the {\it K\"{o}the dual} 
(or {\it associated}) function lattice $E'$ consists of all $y\in S(\Omega,\Sigma,\mu)$ such that
$$
\|y\|_{E'}:=\sup\,\Bigl\{\int_{\Omega}{x(t)y(t)\,d{\mu}}:\;\;
\|x\|_{E}\,\leq{1}\Bigr\}<\infty
$$
(in the case when $E$ is a Banach sequence lattice the integral should be replaced with the sum over $\mathbb{Z}$).

One can easily check that $E'$ is complete with respect to the norm $y\mapsto \|y\|_{E'}$ and $E$ is continuously embedded into its second 
K\"{o}the dual $E''$, with $\|x\|_{E''}\le\|x\|_E$ for $x\in E$. 
A Banach function lattice $E$ {\it has the Fatou property} (or {\it maximal}) if from $x_n\in E,$ $n=1,2,\dots,$ $\sup_{n=1,2,\dots}\|x_n\|_E<\infty$, $x\in S(\Omega,\Sigma,\mu)$ and $x_n\to{x}$ a.e. on $\Omega$ it follows that $x\in E$ and $||x||_E\le \liminf_{n\to\infty}{||x_n||_E}.$ Observe that a Banach function lattice $E$ has the Fatou property if and only if the natural inclusion of $E$ into $E''$ is a surjective isometry \cite[Theorem~6.1.7]{KA}. 

A Banach function lattice $E$ is said to have an {\it order continuous norm} if for every $x\in E$ and any decreasing sequence of sets $A_{n} \in \Sigma$ with $\mu(\bigcap_{n=1}^{\infty} A_n) = 0$ it follows $\|x \chi_{A_{n}} \|_E \rightarrow 0$ as $n \rightarrow \infty$. 

Any K\"{o}the dual function lattice $E'$ is embedded isometrically into (Banach) dual space $E^*$ and $E'=E^*$ if and only if $E$ has an order continuous norm \cite[Corollary~6.1.2]{KA}.

\smallskip

\subsection{Rearrangement invariant spaces.}
\label{prel2}

An important class of Banach function lattices is formed by the so-called rearrangement invariant spaces. We will consider these spaces in the case when underlying measure space is $(0,\infty)$ with the usual Lebesgue measure $m$.

A Banach function lattice $X$ on $(0,\infty)$ is said to be
{\it rearrangement invariant} (in brief, r.i.) (or {\it symmetric}) if, whenever $y\in X$, $x\in S((0,\infty),m)$ and $x^*(t)\le y^*(t)$, $t>0$, we have $x\in X$ and 
$\|x\|_{X} \le \|y\|_{X}$. Here and below, $x^*(t)$ is the right continuous nonincreasing {\it rearrangement} of $|x(s)|$, i.e., 
$$x^{*}(t):=\inf \{ \tau\ge 0:\,m\{s>0:\,|x(s)|>\tau\}\le t \},\;\;t>0.$$
Observe that $x$ and $x^*$ are equimeasurable, that is,
$$
m\{s>0:\,|x(s)|>\tau\}=m\{s>0:\,x^*(s)>\tau\}\;\;\mbox{for all}\;\tau>0.$$ 

Following \cite{LT2}, throughout the paper we will assume that every r.i. space is either separable or has the Fatou property.

For every r.i. space $X$ the following continuous embeddings hold:
$$
L_1\cap L_\infty\subseteq X \subseteq L_1+L_\infty.
$$ 
We denote by $X_0$ the closure of $L_1\cap L_\infty$ in $X$; this set is called the
{\it separable part} of $X$. If $X\ne L_1\cap L_\infty$, then $X_0$ is 
separable. A r.i. space $X$ is separable if and only if $X$ has an order continuous norm (see e.g. \cite[Theorem~4.3.3]{KA}).

Let $X$ be a r.i. space. The function $\varphi _X(t):=\Vert \chi _A \Vert_X$, where $A\subset(0,\infty)$, $m(A)=t$, and $\chi_A$ is the characteristic function of
$A$, is called the {\it fundamental function} of $X$. 

Important examples of r.i. spaces are  the $L^p$-spaces, $1\le
p\le\infty$, and their generalization, the Orlicz spaces (see \cite{KR}, \cite{RR}, \cite{M-89}). Let ${N}$ be an Orlicz function on $[0,\infty)$, i.e., $N$ is a convex continuous increasing function on $[0,\infty)$ with ${N}(0)=0$ and ${N}(\infty)=\infty$. The {\it Orlicz space} $L_N$ consists of all measurable functions $x(t)$ on $(0,\infty)$ for which the Luxemburg norm 
$$
\|x\|_{L_N}:=\inf\Big\{u
>0\,:\,\int_0^\infty N(|x(t)|/u) \,dt\leq 1\Big\}$$ 
is finite.
In particular, if $N(s)=s^p$, $1\le p<\infty$, we obtain the space $L^p=L^p(0,\infty)$ with the usual norm. 

Every Orlicz space $L_N$ has the Fatou property and it is separable if and only if the function $N$ satisfies the {\it $\Delta_2$-condition}, i.e., 
$$
\sup_{u>0}\frac{N(2u)}{N(u)}<\infty.$$

The fundamental function of $L_N$ can be calculated by the formula $\phi_{L_N}(t)=1/N^{-1}(1/t)$, $t>0$, where $N^{-1}$ is the inverse function for $N$.

Another important class of r.i. spaces is formed by the  Lorentz 
spaces. Let $1\le q<\infty$, and let $\psi$ be an increasing concave function on $[0,\infty)$ such that $\psi(0)=0$.
The {\it Lorentz space} $\Lambda_q(\psi)$ consists of all functions $x(t)$ measurable on $(0,\infty)$ and satisfying the condition:
\begin{equation}
\label{eqLor} 
\|x\|_{\Lambda_q(\psi)}:=\left(\int_0^\infty [x^{*}(t)]^{q} 
d\psi(t) \right)^{1/q}<\infty
\end{equation} 
(see \cite{Lo-51}, \cite{KPS}, \cite[p. 121]{LT2})\footnote{There is also a different variant of the Lorentz spaces $\Lambda_{q,\psi}$ endowed with the norm $\|x\|_{\Lambda_{q,\psi}}:=(\int_0^\infty x^*(t)^q\psi(t)^q\,\frac{dt}{t})^{1/q}$ (see e.g. \cite{Sh-72} and  \cite{Ra-92}).}.

For every $1\le q<\infty$ and any concave increasing function $\psi$, $\Lambda_q(\psi)$ is a separable r.i. space with the Fatou property and $\phi_{\Lambda_q(\psi)}(t)=\psi(t)^{1/q}$.

For more detailed information related to r.i. spaces we refer to the books \cite{LT2}, \cite{KPS} and \cite{BSh}.

\smallskip

\subsection{Indices of Banach lattices, r.i. spaces and functions.}
\label{prel3}

Let $E$ be a Banach sequence lattice modelled on $\mathbb{Z}$ such that the shift operator $\tau_na:=(a_{k-n})_{k\in\hj}$, where $a=(a_{k})_{k\in\hj}$, is bounded in $E$ for every $n\in\mathbb{Z}$. If $A\subset\Omega$, where $\Omega$ is $\hj$ or $\gh$, we  denote by $P_A$ the natural projection defined by $P_Af(\omega)=f(\omega)\chi_A(\omega)$, where $f:\,\Omega\to\gh$. Then, denoting $\hj_+=\{k\in\hj:\,k\ge 0\}$ and $\hj_-=\{k\in\hj:\,k\le 0\}$, for each $n\in\hj$ we set:
$$
\|\tau_{n}^0\|_{E\to E}:=\sup_{a\in E:\,a=P_{\hj_-}a}\frac{\|P_{\hj_-}(\tau_{n}a)\|_E}{\|a\|_E}\;\;\mbox{and}\;\;\|\tau_{n}^\infty\|_{E\to E}:=\sup_{a\in E:\,a=P_{\hj_+}a}\frac{\|P_{\hj_+}(\tau_{n}a)\|_E}{\|a\|_E}.
$$
Since these norms are subadditive in $n$, we can define the shift exponents of $E$ by
\begin{align*}
\gamma_E: &= -\lim_{n\to\infty}\frac1n\log_2\|\tau_{-n}\|_{E\to E}, 
&\delta_E: &=\lim_{n\to\infty}\frac{1}{n}\log_2 \|\tau_{n}\|_{E\to E},
\nonumber\\
\gamma_E^0: &= -\lim_{n\to\infty}\frac{1}{n}\log_2 \|\tau_{-n}^0\|_{E\to E}, 
&\delta_E^0&=\lim_{n\to\infty}\frac{1}{n}\log_2 \|\tau_{n}^0\|_{E\to E},
\nonumber\\ \qquad
\gamma_E^\infty&= -\lim_{n\to\infty}\frac{1}{n}\log_2 \|\tau_{-n}^\infty\|_{E\to E}, 
&\delta_E^\infty &=\lim_{n\to\infty}\frac{1}{n}\log_2 \|\tau_{n}^\infty\|_{E\to E}.
\end{align*}

For any $\tau>0$, the dilation operator $\sigma_\tau x(t):=x(t/\tau)$ is bounded in any r.i. space $X$ and $\|\sigma_\tau\|_{X\to X}\le \max(1,\tau)$, $\tau>0$; see e.g. \cite{Boyd} or \cite[Theorem 2.4.4]{KPS}. The numbers
$$
\alpha_X = -\lim_{n\to\infty}\frac{1}{n}\log_2
\|\sigma_{2^{-n}}\|_{X\to X}\quad\mbox{and}\quad  
\beta_X=\lim_{n\to\infty}\frac{1}{n}\log_2 \|\sigma_{2^n}\|_{X\to X},
$$
are called the {\it Boyd indices} of $X$. Similar indices may be introduced also when the norm $\|\sigma_\tau\|_{X\to X}$ is restricted to each of the intervals $[0,1]$ and $[1,\infty)$. Given r.i. space $X$ on $(0,\infty)$, for every $n\in\hj$ we denote
\begin{equation*}
\label{equa16a} 
\|\sigma_{2^n}^0\|_{X\to X}=\sup_{x\in X:\,x=P_{[0,1]}x}\frac{\|P_{[0,1]}(\sigma_{2^n} x)\|_X}{\|x\|_X}
\end{equation*}
and
\begin{equation*}
\label{equa16aa*} 
\|\sigma_{2^n}^\infty\|_{X\to X}=\sup_{x\in X:\,x=P_{[1,\infty)}x}\frac{\|P_{[1,\infty)}(\sigma_{2^n} x)\|_X}{\|x\|_X}.
\end{equation*}
Then, there are the following partial dilation indices:
\begin{align*}
\alpha_X^0&= -\lim_{n\to\infty}\frac{1}{n}\log_2 \|\sigma_{2^{-n}}^0\|_{X\to X}, 
&\beta_X^0&=\lim_{n\to\infty}\frac{1}{n}\log_2 \|\sigma_{2^n}^0\|_{X\to X},
\nonumber\\ \qquad
\alpha_X^\infty&= -\lim_{n\to\infty}\frac{1}{n}\log_2 \|\sigma_{2^{-n}}^\infty\|_{X\to X}, 
&\beta_X^\infty &=\lim_{n\to\infty}\frac{1}{n}\log_2 \|\sigma_{2^n}^\infty\|_{X\to X}.
\end{align*}
We always have $0\le \alpha_X\le\alpha_X^0\le \beta_X^0\le\beta_X\le 1$ and $0\le \alpha_X\le\alpha_X^\infty\le \beta_X^\infty\le\beta_X\le 1$
(see, for instance, \cite[\S\,II.4]{KPS}).

Let $\psi$ be a positive function on $(0,\infty)$. We introduce the  {\it dilation functions}:
$$
M_\psi(t):=\sup_{s>0}\frac{\psi(ts)}{\psi(s)},\quad
M_\psi^0(t):=\sup_{0<s\le\min(1,1/t)}\frac{\psi(ts)}{\psi(s)},
\quad
M_\psi^{\infty}(t):=\sup_{s\ge\max(1,1/t)}\frac{\psi(ts)}{\psi(s)},
$$
and the {\it dilation indices} of  $\psi$ by
\begin{align*}
\mu_\psi&= -\lim_{n\to\infty}\frac{1}{n}\log_2 M_\psi(2^{-n}), 
&\nu_\psi&= \lim_{n\to\infty}\frac{1}{n}\log_2 M_\psi(2^{n}), 
\nonumber\\ \qquad
\mu_\psi^0&= -\lim_{n\to\infty}\frac{1}{n}\log_2 M_\psi^0(2^{-n}), 
&\nu_\psi^0&= \lim_{n\to\infty}\frac{1}{n}\log_2 M_\psi^0(2^{n}), 
\nonumber\\ \qquad
\mu_\psi^\infty&= -\lim_{n\to\infty}\frac{1}{n}\log_2 M_\psi^\infty(2^{-n}), 
&\nu_\psi^\infty&= \lim_{n\to\infty}\frac{1}{n}\log_2 M_\psi^\infty(2^{n}). 
\end{align*}
If $\psi$ is quasi-concave (i.e., $\psi$ is nondecreasing and $\psi(t)/t$ is nonincreasing), then $0\leq\mu_\psi\leq\mu_\psi^0\le \nu_\psi^0\le \nu_\psi\le 1$ and $0\leq\mu_\psi\leq\mu_\psi^\infty\le \nu_\psi^\infty\le \nu_\psi\le 1$ (see \cite[\S2.1.2]{KPS}).

In particular, the fundamental function $\phi_X$ of a r.i. space $X$ is quasi-concave. One can easily check that from the above definitions it follows that  $\alpha_X\le \mu_{\phi_X}$, $\alpha_X^0\le \mu_{\phi_X}^0$, $\alpha_X^\infty\le \mu_{\phi_X}^\infty$, $ \nu_{\phi_X}\le\beta_X$, $\nu_{\phi_X}^0\le\beta_X^0$, and $ \nu_{\phi_X}^\infty\le\beta_X^\infty$.

We will say that a r.i. space $X$ is of {\it fundamental type} whenever the corresponding Boyd indices of $X$ and the dilation indices of $\phi_X$ are equal, i.e.,
\begin{equation}
\label{equa main assump} 
\alpha_X=\mu_{\phi_X}\;,\;\alpha_X^0=\mu_{\phi_X}^0\;,\;\alpha_X^\infty=\mu_{\phi_X}^\infty\;,\;\beta_X=\nu_{\phi_X}\;,\;\beta_X^0=\nu_{\phi_X}^0\;,\;\beta_X^\infty=\nu_{\phi_X}^\infty.
\end{equation}

The most known and important r.i. spaces, in particular, all Lorentz and Orlicz spaces, are of fundamental type\footnote{The first example of a r.i. space of non-fundamental type was constructed by Shimogaki in \cite{Shimo}.}. 

If $X$ is a Lorentz space $\Lambda_q(\psi)$, then $\phi_{X}=\psi^{1/q}$ and hence  
\begin{equation}
\label{Ind of Lor} 
\alpha_{X}=\mu_{\psi^{1/q}}\;,\;\alpha_X^0=\mu_{\psi^{1/q}}^0\;,\;\alpha_X^\infty=\mu_{\psi^{1/q}}^\infty\;,\;\beta_X=\nu_{\psi^{1/q}}\;, \;\beta_X^0=\nu_{\psi^{1/q}}^0\;,\;\beta_X^\infty=\nu_{\psi^{1/q}}^\infty.
\end{equation}

Similarly, substituting the dilation indices of the fundamental function $1/N^{-1}(1/t)$ of an Orlicz space $L_N$ in formulae \eqref{equa main assump}, we can find the dilation indices of $L_N$\footnote{Alternatively, for Orlicz spaces there are used often also the so-called Matuszewska-Orlicz indices $p_N$ and $q_N$ of an Orlicz function $N$ (see e.g. \cite[Proposition~2.b.5]{LT2} or \cite{Ma-85}), which are the reciprocals of the corresponding Boyd indices (i.e., $p_N=1/\beta_{L_N}$ and $q_N=1/\alpha_{L_N}$).}.

Throughout, we denote by $e_n$, $n\in\mathbb{Z}$, the standard unit vectors and by $c_{0,0}$ the set of all finitely supported sequences, i.e., $x=(x_n)_{n\in\mathbb{Z}}\in c_{0,0}$ if ${\rm card}\,\{n:\,x_n\ne 0\}<\infty$.
We write $f\asymp g$ if $cf\leq g\leq Cf$ for some constants $c>0$ and $C>0$ that do not depend on the values of all (or some)  arguments of the functions (quasi-norms) $f$ and $g$. From time to time the value of the constant $C$ may change.

\vskip0.5cm

\section{On a class of Banach sequence lattices generated by r.i.
function spaces.}
\label{dyadic block basis}

Here, we assign to every r.i. function space $X$ a certain Banach sequence lattice $E_X$ such that the sequence $\{\chi_{\Delta_k}\}_{k\in\hj}$, where $\Delta_k:=[2^k,2^{k+1})$, is equivalent in $X$ to the unit vector basis in $E_X$. 

Let $X$ be a r.i. space on $(0,\infty)$. For an arbitrary sequence
$a=(a_k)_{k\in\hj}$ we introduce the following step function
$$
Sa(t):=\sum_{k\in\hj} a_k\chi_{\Delta_k}(t),\;\;t>0.
$$ 
We associate to $X$ the Banach sequence lattice $E_X$ equipped with the norm
$$
\Big\|\sum_{k\in\hj} a_ke_k\Big\|_{E_X}:=\|Sa\|_X.$$
Observe that for all $x\in X$ we have
\begin{equation}
\label{second view} 
\frac12\|x\|_{X}\le \Big\|\sum_{k\in\hj} x^*(2^k)e_k\Big\|_{E_X}\le \|x\|_{X}.
\end{equation}
Indeed, assuming (as we can) that $x=x^*$, by the definition of $E_X$, we come immediately to the right-hand side inequality. Moreover, since $\|\sigma_2\|_{X\to X}\le 2$ \cite[Theorem~2.4.5]{KPS}, we have
$$
\frac12\|x\|_{X}\le \|\sigma_{1/2}x\|_{X}\le\Big\|\sum_{k\in\hj} x(2^{k+1})\chi_{\Delta_{k}}\Big\|_X\le \Big\|\sum_{k\in\hj} x(2^k)e_k\Big\|_{E_X},$$
and the left-hand side inequality in \eqref{second view} follows as well.

The first simple result of this section establishes direct connections between the norms of the dilation operators in $X$ and the shift operators in $E_X$.

\begin{lemma}
\label{dilation}
For every r.i. space $X$ on $(0,\infty)$ we have: 
\begin{equation}\label{equa: 1001}
\|\tau_n\|_{E_X}\le \|\sigma_{2^n}\|_X\;,\;\|\tau_{n}^0\|_{E_X}\le\|\sigma_{2^n}^0\|_X\;,\;\|\tau_{n}^\infty\|_{E_X}\le \|\sigma_{2^n}^\infty\|_X,\;\;\;n\in\mathbb{Z},
\end{equation}
and
\begin{equation}\label{equa: 1001a}
\|\sigma_{2^n}\|_X\le\|\tau_{n+1}\|_{E_X}\;,\;\|\sigma_{2^n}^0\|_X\le\|\tau_{n+1}^0\|_{E_X}\;,\;\|\sigma_{2^n}^\infty\|_X\le\|\tau_{n+1}^\infty\|_{E_X},\;\;\;n\in\mathbb{Z}.
\end{equation}

Hence, $\alpha_X=\gamma_{E_X}$, $\alpha_X^0=\gamma_{E_X}^0$, $\alpha_X^\infty=\gamma_{E_X}^\infty$, $\beta_X=\delta_{E_X}$, $\beta_X^0=\delta_{E_X}^0$ and $\beta_X^\infty=\delta_{E_X}^\infty$.
\end{lemma}
\begin{proof}
Since all inequalities in \eqref{equa: 1001} and \eqref{equa: 1001a} are proved in the same way, we check only the first ones.

For every $n\in\mathbb{Z}$ and $a=(a_{k})_{k\in\hj}$ we have
\begin{equation}
\label{equa102}
S(\tau_na)= \sum_{k\in\hj} a_{k-n}\chi_{\Delta_k}=\sum_{k\in\hj} a_{k}\chi_{\Delta_{n+k}}=\sigma_{2^n}(Sa),
\end{equation}
which yields the first inequality in \eqref{equa: 1001}.

Next, we define the  averaging projection $Q$ by
\begin{equation}\label{projection Q}
Qx(t):=\sum_{k\in\hj} 2^{-k}\int_{\Delta_k}x(s)\,ds\cdot \chi_{\Delta_k}(t),\;\;t>0.
\end{equation}
It is well known that $Q$ is a one norm projection on each r.i. space $X$; see e.g. \cite[\S\,II.3.2]{KPS} (recall that $X$ is assumed to be separable or to have the Fatou property; see Subsection \ref{prel2}). 

Setting $a_x:=(2^{-k}\int_{\Delta_k}x(s)\,ds)_{k\in\hj}$ and comparing the operators $S$ and $Q$, we see that $Sa_x=Qx$. If $x=x^*$, then $x(t)\le Q\sigma_2x(t)$, $t>0$. 
Therefore, since 
$$
a_{\sigma_2x}=\sum_{k\in\hj} 2^{-k}\int_{\Delta_k}x(s/2)\,ds\cdot e_k=\sum_{k\in\hj} 2^{-k}\int_{\Delta_k}x(s)\,ds\cdot e_{k+1}=\tau_1a_x,$$ from \eqref{equa102} it follows that
\begin{eqnarray*}
\|\sigma_{2^n}x\|_X &\le& \|\sigma_{2^n}(Q\sigma_2x)\|_X=\|\sigma_{2^n}(S a_{\sigma_2x})\|_X=\|S(\tau_na_{\sigma_2x})\|_X\\ &=&\|\tau_{n+1}a_x\|_{E_X}\le \|\tau_{n+1}\|_{E_X}\|Qx\|_{X}\le \|\tau_{n+1}\|_{E_X}\|x\|_{X}.
\end{eqnarray*}

To finish the proof, it remains to observe that all the required equalities for the dilation and shift indices follow immediately from inequalities  \eqref{equa: 1001} and  \eqref{equa: 1001a}.
\end{proof}

Calculation of the dilation indices of a r.i. space $X$ is essentially simplified if $X$ is of fundamental type.

\begin{lemma}
\label{dilation1}
Let $X$ be a r.i. space of fundamental type. Then, 
\begin{align*}
\alpha_X &= -\lim_{n\to\infty}\frac{1}{n}\log_2
\sup_{k\in\mathbb{Z}}\frac{s _k}{s _{n+k}}, 
&\beta_X &=\lim_{n\to\infty}\frac{1}{n}\log_2 \sup_{k\in\mathbb{Z}}
\frac{s_{k+n}}{s_{k}},
\nonumber\\
\alpha_X^0&= -\lim_{n\to\infty}\frac{1}{n}\log_2 \sup_{k\le 0}\frac{s_{k-n}}{s_{k}}, 
&\beta_X^0&=\lim_{n\to\infty}\frac{1}{n}\log_2 \sup_{k\le 0}\frac{s_k}{s_{k-n}},
\nonumber\\ \qquad
\alpha_X^\infty&= -\lim_{n\to\infty}\frac{1}{n}\log_2 \sup_{k\ge 0}\frac{s_k}{s_{n+k}}, 
&\beta_X^\infty &=\lim_{n\to\infty}\frac{1}{n}\log_2 \sup_{k\ge 0}\frac{s_{k+n}}{s_k},
\end{align*}
where $s_j:=\|e_j\|_{E_X}$, $j\in\hj$.
\end{lemma}
\begin{proof}
Since $X$ is of fundamental type, by the definitions from Subsection
\ref{prel3} and the quasi-concavity of the fundamental function $\phi_X$, we have
\begin{align*}
\alpha_X &= -\lim_{n\to\infty}\frac{1}{n}\log_2
\sup_{k\in\mathbb{Z}}\frac{\phi_X(2^k)}{\phi_X(2^{n+k})}, 
&\beta_X &=\lim_{n\to\infty}\frac{1}{n}\log_2 \sup_{k\in\mathbb{Z}}
\frac{\phi_X(2^{n+k})}{\phi_X(2^{k})},
\nonumber\\
\alpha_X^0&= -\lim_{n\to\infty}\frac{1}{n}\log_2 \sup_{k\le 0}\frac{\phi_X(2^{k-n})}{\phi_X(2^{k})}, 
&\beta_X^0&=\lim_{n\to\infty}\frac{1}{n}\log_2 \sup_{k\le 0}\frac{\phi_X(2^{k})}{\phi_X(2^{k-n})},
\nonumber\\ \qquad
\alpha_X^\infty&= -\lim_{n\to\infty}\frac{1}{n}\log_2 \sup_{k\ge 0}\frac{\phi_X(2^{k})}{\phi_X(2^{k+n})}, 
&\beta_X^\infty &=\lim_{n\to\infty}\frac{1}{n}\log_2 \sup_{k\ge 0}\frac{\phi_X(2^{k+n})}{\phi_X(2^{k})},
\end{align*}
Since $\phi_X(2^j)=\|\chi_{\Delta_j}\|_X=\|e_j\|_{E_X}$ for all $j\in\hj$, we obtain the desired result. 
\end{proof}

The next result of this section is in fact known (see \cite[Proposition~5.1]{Kal}). It shows that the Banach lattice $E_X$ associated to a r.i. space $X$ can be defined also in a different way, when one starts from a Banach sequence lattice satisfying inequalities of type \eqref{second view} and having positive lower shift exponent. Because the proof of this assertion in \cite{Kal} is only outlined, for convenience of the reader, we provide here it with a  detailed proof. 

\begin{prop}
\label{prop:latttice and symm}
Let $E$ be a Banach sequence lattice, with $\gamma_E>0$,  and let $X$ be a r.i. space satisfying
\begin{equation}
\label{equa10} 
\|x\|_{X}\asymp \Big\|\sum_{k\in\hj} x^*(2^k)e_{k}\Big\|_E.
\end{equation}
Then $E_X=E$ (with equivalence of norms).
\end{prop}
\begin{proof}
Let us note first that for each nonincreasing and nonnegative sequence $a=(a_{k})_{k\in\hj}$ the norms $\|a\|_{E_X}$ and $\|a\|_{E}$ are equivalent (with the equivalence constant from \eqref{equa10}). Indeed, since $(\sum_{i\in\hj} a_i\chi_{\Delta_i})^*(2^k)=a_k$, $k\in\hj$, by the definition of $E_X$ and \eqref{equa10}, we have
$$
\|a\|_{E_X}=\Big\|\sum_{k\in\hj} a_{k}\chi_{\Delta_k}\Big\|_X\asymp\Big\|\sum_{k\in\hj} a_{k}e_{k}\Big\|_E=\|a\|_{E}.$$
Assume now that $a=(a_{k})_{k\in\hj}\in E_X$ is arbitrary. The condition $\gamma_E>0$ ensures that for some $\epsilon>0$ and $C>0$ we have  
\begin{equation}
\label{equa10a} 
\|\tau_{-j}\|_{E\to E}\le C2^{-\epsilon j},\;\;j\in\mathbb{N}.
\end{equation}
Consequently, since the sequence $\left(\max_{j\ge 0}|\tau_{-j}a|\right)_k$, $k=1,2,\dots$, is nonincreasing and
$$
|a_k|\le (\max_{j\ge 0}|\tau_{-j}a|)_k,\;\;k\in\hj,$$
by the above observation, we have
\begin{eqnarray*}
\|a\|_{E_X} &\le& \left\|\left(\max_{j\ge 0}|\tau_{-j}a|\right)_k\right\|_{E_X}\asymp 
\left\|\left(\max_{j\ge 0}|(\tau_{-j}a|\right)_k\right\|_{E}\\ &\le& \Big\|\Big(\sum_{j=0}^\infty |\tau_{-j}a|\Big)_k\Big\|_{E}\le 
\sum_{j=0}^\infty \|\tau_{-j}\|_{E\to E} \|a\|_E\\&\le& 
\sum_{j=0}^\infty 2^{-\epsilon j} \|a\|_E = C'\|a\|_E. 
\end{eqnarray*}

Conversely, for each $a=(a_{k})_{k=1}^\infty\in E$ 
\begin{eqnarray*}
\|a\|_{E} &\le& \left\|\left(\max_{j\ge 0}|\tau_{-j}a|\right)_k\right\|_{E}\asymp 
\left\|\left(\max_{j\ge 0}|\tau_{-j}a|\right)_k\right\|_{E_X}\nonumber\\ &=& \Big\|\sum_{k\in\hj} \left(\max_{j\ge 0}|\tau_{-j}a|\right)_k\chi_{\Delta_k}\Big\|_{X}\nonumber\\&\le&
\Big\|\sum_{k\in\hj}\sum_{j=0}^\infty (|\tau_{-j}a|)_k\chi_{\Delta_k}\Big\|_{X}\nonumber\\&\le&
\sum_{j=0}^\infty\Big\|\sum_{k\in\hj} (|\tau_{-j}a|)_k\chi_{\Delta_k}\Big\|_{X}. 
\end{eqnarray*}
This inequality combined together with the observation that
$$
\sum_{k\in\hj} (|\tau_{-j}a|)_k\chi_{\Delta_k} =
\sum_{k\in\hj}|a_{k+j}|\chi_{\Delta_k}=
\sum_{k\in\hj}|a_{k}|\chi_{\Delta_{k-j}}=\sigma_{2^{-j}}\Big(\sum_{k\in\hj} |a_k|\chi_{\Delta_k}\Big)
$$
yields
\begin{equation}
\label{equa12} 
\|a\|_{E} \le \sum_{j=0}^\infty\Big\|\sigma_{2^{-j}}\Big(\sum_{k\in\hj} |a_k|\chi_{\Delta_k}\Big)\Big\|_{X}.
\end{equation}

Next, we claim that there exists a constant $C > 0$ such that for every $x\in X$ and all integers $j\ge 0$
\begin{equation}
\label{equa13} 
\|\sigma_{2^{-j}}x\|_X\le C\|\tau_{-j}\|_{E\to E}\|x\|_X.
\end{equation}
Since $(\sigma_{\tau}x)^*=\sigma_{\tau}(x^*)$, $\tau>0$, we can assume that $x=x^*$. Then, by \eqref{equa10}, we have
\begin{eqnarray*}
\|\sigma_{2^{-j}}x\|_X &\asymp& \Big\|\sum_{k\in\hj}x(2^{k-j})e_{k}\Big\|_E=\Big\|\tau_{{-j}}\Big(\sum_{k\in\hj}x(2^{k})e_{k}\Big)\Big\|_E\\ &\le& \|\tau_{{-j}}\|_{E\to E}\Big\|\sum_{k\in\hj}x(2^{k})e_{k}\Big\|_E\le C\|\tau_{{-j}}\|_{E\to E}\|x\|_X,
\end{eqnarray*}
and \eqref{equa13} follows.

Applying estimate \eqref{equa13} to a function $x=\sum_{k\in\hj} a_k\chi_{\Delta_k}$, we get
\begin{eqnarray*}
\Big\|\sigma_{2^{-j}}\Big(\sum_{k\in\hj} a_k\chi_{\Delta_k}\Big)\Big\|_X &\le& C\|\tau_{-j}\|_{E\to E}\Big\|\sum_{k\in\hj} a_k\chi_{\Delta_k}\Big\|_X\nonumber\\
&=& C\|\tau_{-j}\|_{E\to E}\|a\|_{E_X}.
\label{equa14}
\end{eqnarray*}
Therefore, from \eqref{equa12} and \eqref{equa10a} it follows 
$$
\|a\|_E \le C\sum_{j=0}^\infty\|\tau_{-j}\|_{E\to E}\|a\|_{E_X}\le 
C\sum_{j=0}^\infty 2^{-\epsilon j}\|a\|_{E_X}\le C\|a\|_{E_X},
$$
completing the proof.
\end{proof}

From Proposition \ref{prop:latttice and symm} and Lemma  \ref{dilation}
it follows

\begin{cor}
\label{cor:basic}
Let $X$ be a r.i. space such that equivalence \eqref{equa10} holds for some Banach sequence lattice $E$ such that $\gamma_E>0$. Then, $\alpha_X>0$.
\end{cor}

\vskip0.5cm

\section{Some connections between dilations in r.i. function spaces and shifts in corresponding sequence lattices.}

Let ${{\sigma}}x(t)=x(t/2)$, $t>0$ and ${{{\sigma}}_\lambda}:={{{\sigma}}}-\lambda I$, $\lambda>0$. We show that certain properties of the operators ${{{\sigma}}_\lambda}$ in a r.i. function space $X$ and ${\tau_\lambda}={\tau}-\lambda I$ in the corresponding Banach sequence lattice $E_X$, which was introduced in the preceding section, are connected in a direct way.

Putting as above $Sa(t)=\sum_{k\in\hj} a_k\chi_{\Delta_k}(t)$, where $\Delta_k=[2^k,2^{k+1})$, $k\in\hj$, for every sequence $a=(a_k)_{k\in\hj}$ we have
\begin{equation}
\label{eq15} 
\begin{split}
{\sigma}_\lambda Sa(t) &=\sum_{k\in\hj}
a_k\chi_{\Delta_k}(t/2)-\lambda \sum_{k\in\hj}
a_k\chi_{\Delta_k}(t)
\\
&=\sum_{k\in\hj} a_k\chi_{\Delta_{k+1}}(t)-\lambda
\sum_{k\in\hj} a_k\chi_{\Delta_k}(t)
\\
&=\sum_{k\in\hj} \big (\tau_\lambda a\big)_k\chi_{\Delta_k}(t)=S{\tau_\lambda}a(t),\;\;t>0.
\end{split}
\end{equation}
Hence, appealing to the definition of $E_X$, we obtain
\begin{equation}
\label{main equiv}
 \|{{{\sigma}}_\lambda} {Sa}\|_{{X}}=\|S{\tau_\lambda}a\|_{{X}}=\|{\tau_\lambda}a\|_{{E_X}}.
\end{equation}

\begin{prop}\label{Proposition 1}
For every r.i. space $X$ on $(0,\infty)$ and any $\lambda>0$ we have the following: 

{\rm(i)} ${{{\sigma}}_\lambda} $ is injective in $X$ if and only if $\tau_\lambda$ is injective in ${E_X}${\rm;}

{\rm(ii)} if ${{{\sigma}}_\lambda} $ is closed in $X$, then $\tau_\lambda$ is closed in ${E_X}${\rm;}

{\rm(iii)} if $\tau_\lambda$ is injective and closed in ${E_X}$, then ${{{\sigma}}_\lambda} $ is closed in ${X}$.
\end{prop}

\begin{proof}

(i) The fact that the injectivity of ${\sigma}_\lambda $ implies that of $\tau_\lambda$ is a direct consequence of \eqref{main equiv}. 

Let us prove the converse. For a contradiction, suppose that there is $x\in X$, $x\ne 0$, with ${\sigma}_\lambda x=0$. Since 
$|{\sigma}_{\lambda} x|\ge \big|\,|x(t/2)|- \lambda|x(t)|\,\big|$, we may assume that $x(t)$ is nonnegative. Outside a set of zero measure, we have
\begin{equation}
\label{eq17} x(t/2)=\lambda\cdot x(t),\;\;t>0,
\end{equation}
whence it follows that
\begin{align}
\int_{\Delta_k}x(t)\,dt&=\frac{1}{\lambda} \int_{\Delta_k}x(t/2)\,dt=
\frac{2}{\lambda} \int_{\Delta_{k-1}}x(t)\,dt
\\
\intertext{or}
\int_{\Delta_k}x(t)\,dt&= \Big(\frac{2}{\lambda}\Big)^k
\int_{\Delta_{0}}x(t)\,dt,\quad k\in\hj.
\label{eq17.2}
\end{align}

Let $Q$ be the one norm averaging operator defined by \eqref{projection Q}. If $a_x=(2^{-k}\int_{\Delta_k}x(s)\,ds)_{k\in\hj}$, then $Qx=Sa_x$ and hence, by the definition of $E_X$, the sequence $a_x$ belongs to ${E_X}$.  Thus, from \eqref{eq17.2} it follows that the sequence $a=(a_k)_{k\in\hj}$, with $a_k:=\lambda^{-k}\int_{{\Delta_0}}x(s)\,ds$, also belongs to this space. At the same time, it is easily seen that $\tau_{\lambda}a=0$. Since $x\ge 0$ and $x\ne 0$, by \eqref{eq17} 
we deduce that $a\ne 0$. Therefore, $\tau_{\lambda}$ is not injective, which contradicts the hypothesis.

\smallskip
(ii) If $a^n=(a^n_k)_{k\in\hj}\in {E_X}$, $n=1,2,\dots$, and $\tau_{\lambda}a^n\to b=(b_k)$ in ${E_X}$, then, by \eqref{main equiv}, the functions $\{{\sigma}_{\lambda}S{a^n}\}$ form a Cauchy sequence in the space $X$. By assumption, ${\sigma}_{\lambda}S{a^n}\to y:= {\sigma}_{\lambda}x$, where $x\in X$. On the other hand, from \eqref{eq15} it follows that ${\sigma}_{\lambda}S{a^n}=S \tau_\lambda a^n\to Sb$. Therefore, ${\sigma}_{\lambda}x=Sb$, that is,
$$
x(t/2)-\lambda x(t)=\sum_{k\in \hj}b_k\chi_{\Delta_k}(t),\;\;t>0.$$
Integrating over $\Delta_k$, $k\in\hj$, gives then
$$
2\int_{\Delta_{k-1}}x(s)\,ds-\lambda \int_{\Delta_k}x(s)\,ds=2^kb_k,\;\;k\in\hj,$$
whence $\tau_\lambda a_x=b$. As $a_x\in {E_X}$, we conclude that
$b\in{\rm Im}\,\tau_{\lambda}$. Consequently, $\tau_{\lambda}$ is a closed operator.

\smallskip

(iii) 
To get a contradiction, assume that the operator ${\sigma}_{\lambda}$ is not closed. Then there exists a 
sequence $\{x_{n}\}\subset {X}$ with the properties:
\begin{equation}
\label{eq19.1}
\|x_{n}\|_{{X}}=1,\;\;n=1,2,\dots,\quad\mbox{and}\quad \|
{\sigma}_{\lambda}x_{n}\|_{{X}}\to 0.
\end{equation}
Since ${X}$ is separable or has the Fatou property, we have
$$
\|{\sigma}_{\lambda}x_{n}\|_X\ge \|{\sigma}(x_{n}^*)-\lambda x_{n}^*\|_X=\|{\sigma}_{\lambda}(x_{n}^*)\|_X
$$ 
(see e.g. \cite[Theorems II.4.9, II.4.10 and Lemma~II.4.6]{KPS}).
Hence, we may assume that each of the functions $x_{n}$ in \eqref{eq19.1} is nonnegative and nonincreasing. 

Next, if $Q$ is the averaging projection defined by \eqref{projection Q}, then for every $x\in X$ we have
\begin{align*}
Q{\sigma}_{\lambda}x &= \sum_{k\in\hj}
2^{-k}\int_{\Delta_k}{\sigma}_{\lambda}x(s)\,ds \cdot\chi_{\Delta_k}
\\
&= \sum_{k\in\hj} 2^{-k}\Big(\int_{\Delta_k}x(s/2)\,ds-
\lambda\int_{\Delta_k}x(s)\,ds\Big)
\cdot\chi_{\Delta_k}
\\
&= \sum_{k\in\hj}
2^{-k}\Big(2\int_{\Delta_{k-1}}x(s)\,ds-
\lambda\int_{\Delta_k}x(s)\,ds\Big)\cdot\chi_{\Delta_k}
\\
&= \sum_{k\in\hj} \big(a_x)_{k-1}-\lambda (a_x)_k\big)
\cdot\chi_{\Delta_k}= \sum_{k\in\hj} \big( \tau_\lambda
a_x\big)_k\cdot\chi_{\Delta_k}\\
&= S\tau_\lambda a_x,
\end{align*}
whence $\|\tau_\lambda a_x\|_{E_X}=\|Q{\sigma}_{\lambda}x\|_X$.
As $Q$ is bounded on ${X}$, from this equality and \eqref{eq19.1} it follows that $ \|\tau_\lambda a_{x_n}\|_{{E_X}}\to 0$ as $n\to\infty$. 
On the other hand, applying successively the definition of the norm in $E_X$, the monotonicity of each function $x_{n}$, the inequality $\|\sigma_2\|_{X\to X}\le 2$ and finally \eqref{eq19.1} once more, we get 
\begin{align*}
\|a_{x_n}\|_{{E_X}}&=\Big\| \sum_{k\in\hj}
2^{-k}\int_{\Delta_k} x_n(s)\,ds\cdot
\chi_{\Delta_k}\Big\|_{X} 
\\
&\ge \Big\|\sum_{k\in\hj}
x_n(2^{k+1})\cdot\chi_{\Delta_k}\Big\|_{X}\\
&\ge \|\sigma_{1/2}x_n\|_X\ge \frac{1}{2}\|x_n\|_X= \frac{1}{2}.
\end{align*}
Summarizing all, we see that the operator $\tau_\lambda$ fails to be an isomorphic embedding in the space ${E_X}$. Since it is injective, this means that $\tau_\lambda$ is not closed in this space, which contradicts the assumption.
\end{proof}

\vskip0.5cm

\section{A description of approximative eigenvalues of the shift operator in Banach sequence lattices.}

Let $E$ be a Banach sequence lattice such that the shift operator $\tau(a_k)=(a_{k-1})$ and its inverse $\tau_{-1}(a_k)=(a_{k+1})$  are bounded in $E$. Denote $s_k:=\|e_k\|_E$, where $e_k$, $k\in \hj$, are elements of the unit vector basis. Assume that $E$ is a lattice of "fundamental type", i.e.,  for every $n\in\mathbb{Z}$, the norms $\|\tau_{n}\|_{E\to E}$, $\|\tau_{n}^0\|_{E\to E}$ and $\|\tau_{n}^\infty\|_{E\to E}$ (see Subsection \ref{prel3}) can be calculated (up to equivalence with a constant independent of $n$) when these operators are restricted to the sets $\{e_k\}_{k\in\hj}$, $\{e_k\}_{k\le 0}$ and $\{e_k\}_{k\ge 0}$ respectively. More explicitly, 
\begin{align}
\label{equa16a} 
\|\tau_{-n}\|_{E\to E} &\asymp\sup_{k\in\hj}\frac{s_k}{s_{n+k}},\;\;n\in\mathbb{N}, &\|\tau_{n}\|_{E\to E}&\asymp\sup_{k\in\hj}\frac{s_{k+n}}{s_{k}},\;\;n\in\mathbb{N},\nonumber\\ 
\|\tau_{-n}^0\|_{E\to E}&\asymp\sup_{k\ge 0}\frac{s_k}{s_{n+k}},\;\;n\in\mathbb{N},&\|\tau_{n}^0\|_{E\to E}&\asymp\sup_{k\ge 0}\frac{s_{k+n}}{s_{k}},\;\;n\in\mathbb{N},\nonumber\\ 
\|\tau_{-n}^\infty\|_{E\to E}&\asymp\sup_{k\le 0}\frac{s_{k-n}}{s_{k}},\;\;n\in\mathbb{N},&\|\tau_{n}^\infty\|_{E\to E}&\asymp\sup_{k\le 0}\frac{s_k}{s_{k-n}},\;\;n\in\mathbb{N}.
\end{align}
Consequently, we have
\begin{align}
\gamma &= -\lim_{n\to\infty}\frac{1}{n}\log_2
\sup_{k\in\mathbb{Z}}\frac{s _k}{s _{n+k}}, 
&\delta&=\lim_{n\to\infty}\frac{1}{n}\log_2 \sup_{k\in\mathbb{Z}}
\frac{s_{k+n}}{s _{k}},
\nonumber\\
{\gamma^0}&= -\lim_{n\to\infty}\frac{1}{n}\log_2 \sup_{k\le 0}\frac{s_{k-n}}{s_{k}}, 
&\delta^0&=\lim_{n\to\infty}\frac{1}{n}\log_2 \sup_{k\le 0}\frac{s_k}{s_{k-n}},
\nonumber\\ \qquad
\gamma^\infty&= -\lim_{n\to\infty}\frac{1}{n}\log_2 \sup_{k\ge 0}\frac{s_k}{s_{n+k}}, 
&{\delta^\infty} &=\lim_{n\to\infty}\frac{1}{n}\log_2 \sup_{k\ge 0}\frac{s_{k+n}}{s_k}
\label{equa16b}
\end{align}
(for brevity, in this section we set $\gamma:=\gamma_E$, $\gamma^0:=\gamma_E^0$,  $\gamma^\infty:=\gamma_E^\infty$, $\delta:=\delta_E$, $\delta^0:=\delta_E^0$, $\delta^\infty:=\delta_E^\infty$).

It is easily seen that 
\begin{equation}
\label{eq19.10}
\gamma\le {\gamma^0}\le{\delta^0}\le\delta\;\;\mbox{and}\;\;\gamma\le{\gamma^\infty}\le{\delta^\infty}\le\delta.
\end{equation}
Moreover, reasoning in the same way as in \cite[Lemma 1]{ASun}, one can show that $\gamma=\min({\gamma^0},{\gamma^\infty})$ and 
$\delta=\max({\delta^0},{\delta^\infty})$.

We will be interested in properties of the operators $\mn=\tau-\lambda I$, $\lambda>0$, where $I$ is the identity in $E$. We show that, in terms of the above exponents, it is possible to determine the set of all parameters $\lambda$, for which $\mn$ is an
isomorphic embedding in $E$. Thereby, as consequence, we will identify the set of all approximate eigenvalues of the operator $\tau$. 

Let $P_+$ and $P_-$ denote the norm one projections on $E$ defined by
$$
P_+(a_k):=\sum_{k\ge 1} a_ke_k\;\;\mbox{and}\;\;P_-(a_k):=\sum_{k\le -1} a_ke_k,$$
and let $r(T)$ stand for the spectral radius of an operator $T$ bounded in $E$.

\begin{lemma}
\label{L1-new}
We have $r(\tau)=2^\delta$, $r(\tau_{-1})=2^{-\gamma}$, $r(\tau P_-)=2^{\delta _0}$ and $r(\tau_{-1}P_+)=\,2^{-\gamma _\infty}$.
\end{lemma}
\begin{proof}
Since $\tau^n=\tau_n$ and $\tau_{-1}^n=\tau_{-n}$, $n\in\mathbb{N}$, then by assumption \eqref{equa16a} we have
$$
\|\tau^n\|_{E\to E}\asymp\sup_{k\in\hj}\frac{s_{k+n}}{s_{k}}\;\;\mbox{and}\;\;\|\tau^{-n}\|_{{E}\to{E}}\asymp\sup_{k\in\mathbb{N}}\frac{s_k}{s_{k+n}},\;\;n\in\mathbb{N}.$$
Therefore, from \eqref{equa16b} it follows that
$$
r(\tau)=\lim_{n\to\infty}\left(\sup_{k\in\hj}\frac{s_{k+n}}{s_{k}}\right)^{1/n}=2^\delta\;\;\mbox{and}\;\;r(\tau_{-1})=\lim_{n\to\infty}\left(\sup_{k\in\mathbb{Z}}\frac{s_k}{s_{k+n}}\right)^{1/n}=2^{-\gamma}.$$

Similarly, since
$$
(\tau P_-)^n(a_k)=\sum_{k\le -n} a_ke_{k+n}\;\;\mbox{and}\;\;(\tau_{-1}P_+)^n(a_k)=\;\sum_{k\ge n} a_ke_{k-n},\;\;n\in\mathbb{N},$$
by \eqref{equa16a}, it holds
$$
\|(\tau P_-)^n\|_{{E}\to{E}}=\sup_{k\le 0}\frac{s_k}{s_{k-n}}\;\;\mbox{and}\;\;\|(\tau_{-1}P_+)^n\|_{{E}\to{E}}=\sup_{k\ge 0}\frac{s_k}{s_{k+n}}.$$
As a result, in view of \eqref{equa16b}, we have
$$
r(\tau P_-)=\lim_{n\to\infty}\left(\sup_{k\le 0}\frac{s_k}{s_{k-n}}\right)^{1/n}=2^{\delta^0}\;\;\mbox{and}\;\;r(\tau_{-1}P_+)=\lim_{n\to\infty}\left(\sup_{k\ge 0}\frac{s_k}{s_{k+n}}\right)^{1/n}=2^{-\gamma^\infty}.$$
\end{proof}

\begin{theor}\label{L2-new}
Suppose a separable Banach sequence lattice $E$ satisfies condition \eqref{equa16a}. Then, we have

(i) if ${\gamma^\infty}\le {{\delta^0}}$, then the operator $\mn$ is an isomorphism in ${E}$ if and only if $\lambda\in (0,2^\gamma)\cup (2^{\delta},\infty)$;

(ii) if ${\gamma^\infty}>{{\delta^0}}$, then $\mn$ is an isomorphism in ${E}$ if and only if $\lambda\in (0,2^\gamma)\cup (2^{{\delta^0}},2^{{\gamma^\infty}})\cup (2^{\delta},\infty)$.
 
Moreover, if $ \lambda\in (0,2^\gamma)\cup (2^{\delta},\infty)$, then ${\rm Im}\,\mn=E$; if $\lambda\in (2^{{\delta^0}},2^{{\gamma^\infty}})$, then ${\rm Im}\,\mn$ is the closed subspace of codimension $1$ in $E$ consisting of all $(a_k)_{k\in\hj}\in{E}$ with
\begin{equation}
\label{equa15a} 
{\sum_{k\in \mathbb{Z}}} \lambda^k a_k=0.
\end{equation}
\end{theor}

\begin{proof}
First, if $\lambda\in (2^\delta,\infty)$, Lemma \ref{L1-new} implies that $\lambda >r(\tau)$, whence $\mn$ is an isomorphism from ${E}$ onto ${E}$. The same holds also in the case when $\lambda\in (0,2^\gamma)$. Indeed, again by Lemma \ref{L1-new} we have $r(\tau_{-1})<\lambda^{-1},$ and hence the operator $\tau_{-1}-\lambda^{-1} I$ is an isomorphism from ${E}$ onto ${E}$. Thereby, the desired result follows, because 
\begin{equation*}
\label{eu1} 
\mn=\lambda \tau (\lambda^{-1} I-\tau_{-1})
\end{equation*}
and $\tau$ is isomorphic in $E$.

Next, let us determine the possible form of the image $\op$. Since
$$
\mn\left(\sum_{i=0}^{n-1} \lambda^{n-1-i}e_i\right)=e_{n}-\lambda^{n}e_0,\;\;n\in\hj,$$
we get $e_{n}-\lambda^{n}e_0\in\op$ for all $n\in\hj.$ Suppose $f_\lambda$  is a linear functional vanishing at $\op$. Then $f_\lambda (e_{n})=\lambda^{n}f_\lambda (e_{0})$, $n\in\hj$. Therefore, we may assume that $f_\lambda$  corresponds to the sequence $(\lambda^{n})_{n\in\hj}.$ 

Observe that $f_\lambda(\tau_\lambda a)=0$ for every $a=(a_k)_{k\in\hj}\in c_{00}$. Indeed, for each $k\in \mathbb{Z}$ we have
$$ 
\tau_\lambda e_k =\tau_\lambda \left(\sum_{i=0}^{k-1} \lambda^{k-i-1}e_i-\lambda\sum_{i=0}^{k-2} \lambda^{k-i-2}e_i\right)=e_{k}-\lambda^{k}e_0-\lambda(e_{k-1}-\lambda^{k-1}e_0)=e_{k}-\lambda e_{k-1}.
$$
Hence,
$$
f_\lambda(\tau_\lambda e_k)=f_\lambda(e_{k})-\lambda f_\lambda(e_{k-1})=0,\;\;k\in\hj,$$
and the claim is proved. Thus, as $E$ is a separable space, if $f_\lambda\in{E} ^*$, we have $f_\lambda(a)=0$ for all $a\in\op$. So, the Hahn-Banach theorem implies that
\begin{equation}
\label{eq16b} 
\overline{\op}=\zw,
\end{equation}
where $\zw$ is the closed subspace of codimension $1$ in $E$ consisting of all $(a_k)_{k=1}^\infty\in{E}$ satisfying \eqref{equa15a}.
In turn, the condition $f_\lambda\in{E} ^*$ is equivalent to the fact that 
\begin{equation}
\label{eq16a} 
{\sum_{k\in \mathbb{Z}}} \lambda^{k}a_k <\infty,\;\;\mbox{for every}\;\;(a_k)_{k\in \mathbb{Z}}\in E.
\end{equation}
On the other hand, if $f_\lambda\not\in{E} ^*$, we have
\begin{equation}
\label{eq16c} 
\overline{\op}={E}.
\end{equation}

Assume now that ${\delta^0}<{\gamma^\infty}$ and $\lambda\in (2^{{\delta^0}},2^{{\gamma^\infty}})$. Let us show that $f_\lambda\in{E} ^*$ and $\op$ is closed.

Choose $\varepsilon>0$ so that ${\delta^0} +\varepsilon <\log_2\lambda <{\gamma^\infty} -\varepsilon$. Then, by the definition of indices ${\delta^0}$ and ${\gamma^\infty}$, for some $C>0$ we get
$$
\sup_{k\ge 0}\frac{s_k}{s_{k+n}}\,\le\,C2^{-n({\gamma^\infty} -\varepsilon)}\;\;\mbox{and}\;\;\sup_{k\le 0}\frac{s_k}{s_{k-n}}\le C2^{n({\delta^0} +\varepsilon)},\;\;n=0,1,2,\dots,$$
whence
\begin{equation}
\label{eu2} 
s_n ^{-1}\le s_0^{-1}C2^{-n({\gamma^\infty} -\varepsilon)}\;\;\mbox{and}\;\;s_{-n}^{-1}\le s_0 ^{-1} C2^{n({\delta^0} +\varepsilon)},\;\;n\in\mathbb{N}.
\end{equation}
Therefore,
$$
\sum_{n\in\fg} \lambda^{n} s_n ^{-1}\le s_0^{-1}C\left(\sum_{n=0}^\infty (\lambda 2^{-{\gamma^\infty} +\varepsilon})^n+\sum_{n=1}^\infty (\lambda^{-1} 2^{{\delta^0} +\varepsilon})^n\right)<\infty.$$
On the other hand, since $E$ is a Banach lattice, for every $a=(a_k)_{k\in\hj}\in E$ we have $|a_k|\|e_k\|_E\le\|a\|_E$, $k\in\hj$, i.e.,
\begin{equation}
\label{eu3} 
|a_k|\le\|a\|_Es_k^{-1},\;\;k\in\mathbb{Z}.
\end{equation}
Combining the last inequalities, we get
$$
\sum_{n\in\hj} \lambda^{n}a_n\le \|a\|_E\sum_{n\in\hj} \lambda^{n}s_n^{-1}<\infty.$$ 
Thus, condition \eqref{eq16a} is fulfilled for each $a=(a_k)_{k\in\hj}\in E$, yielding $f_\lambda\in{E} ^*.$ 

In order to prove that $\op$ is closed, let us represent ${E}$ as follows:
$$
{E}=E_-+E_0+E_+,$$
where $E_-=[\{e_n\}_{n\le -1}]_{E},$ $E_0=[\{e_0\}]_{E}$ and $E_+=[\{e_n\}_{n\ge 1}]_{E}$ $([A]_{E}$ is the closed linear span of a set $A$ in $E).$
Since 
$$
\op=\mn(E_-)+\mn(E_0)+\mn(E_+),$$
$\mn(E_-)\cap\mn(E_+)=\{0\}$ and the subspace $\mn(E_0)$ is one-dimensional, it suffices to prove only that the sets $\mn(E_-)$
and $\mn(E_+)$ are closed.

Since $2^{{\delta^0}} <\lambda<2^{{\gamma^\infty}}$, $r(\tau P_-)=2^{{\delta^0}}$ and $r(\tau_{-1}P_+)=2^{-{\gamma^\infty}}$ by Lemma \ref{L1-new}, it follows that the operators $\tau P_--\lambda I$ and $\tau_{-1}P_+-\lambda^{-1} I$ map ${E}$ onto ${E}$ isomorphically. Therefore, noting that the subspace $(\tau P_--\lambda I)(E_-)$ is closed in $E$ and $\mn(E_-)= (\tau P_--\lambda I)(E_-)$, we conclude that $\mn(E_-)$ is closed as well. Similarly, because 
$$
\mn(E_+)=-\lambda \tau(\tau_{-1}P_+-\lambda^{-1} I)(E_+)$$ 
and $\tau$ is an isomorphism in ${E}$, it follows that  $\mn(E_+)$ is closed. 

Furthermore, the operator $\mn$ is injective if  
$\lambda\in (0,2^{{\gamma^\infty}})\cup (2^{{\delta^0}},\infty)$.
Indeed, the equation $\mn a=0$ with $a=(a_n)_{n\in\hj}\in E$ implies that $a_{n-1}=\lambda a_n$, i.e., $a_n=\lambda^{-n}a_0$ for all $n\in\hj.$
Assuming that $a\ne 0$, we infer that $(\lambda^{-n})_{n\in\hj}\in E.$

If, for instance, $\lambda<2^{{\gamma^\infty}}$, then combining the first inequality in \eqref{eu2}, where $0<\varepsilon<{\gamma^\infty}-\log_2\lambda$, together with \eqref{eu3}, we deduce that
$$
\lambda^{-k}\le\|(\lambda^{-n})_{n\in\hj}\|_Es_k^{-1}\le s_0^{-1}C\|(\lambda^{-n})_{n\in\hj}\|_E2^{-k({\gamma^\infty} -\varepsilon)},\;\;k\in\mathbb{N},$$
and hence $\lambda\ge 2^{{\gamma^\infty} -\varepsilon}$, which contradicts the choice of $\varepsilon$. Similar reasoning by using the second inequality in \eqref{eu2}, where $0<\varepsilon<\log_2\lambda-{\delta^0}$, shows that the assumption $(\lambda^{-n})_{n\in\hj}\in E$ implies a contradiction if $\lambda>2^{{\delta^0}}$.

Thus, if ${\delta^0}<{\gamma^\infty}$ and $\lambda\in (2^{{\delta^0}},2^{{\gamma^\infty}})$, the operator $\mn$ maps isomorphically $E$ onto  the closed subspace of codimension $1$ in $E$ consisting of all $(a_k)_{k\in\hj}\in{E}$ satisfying condition \eqref{equa15a}.

It remains to establish the necessity of conditions in $(i)$ and $(ii)$. To this end, we assume that the operator $\mn$ is an isomorphic mapping in $E$. Then, there exists $c > 0$ such that for all $x\in{E}$
\begin{equation}
\label{equa18} 
\|\mn x\|_{E}\ge c\|x\|_{E}.
\end{equation}
We prove the following implication:
\begin{equation}
\label{eu21a} 
\lambda\not\in (0,2^\gamma]\cup [2^{\delta},\infty)\;\;\mbox{and inequality}\;\;\eqref{equa18}\;\mbox{holds}\;\;\Longrightarrow\;\; \lambda\in [2^{{\delta^0}},2^{{\gamma^\infty}}].
\end{equation}

Let $n\in\mathbb{N}$ and $k\in\hj$ be arbitrary (they will be fixed later). We put
$a:=(I+\lambda^{-1}\tau+\dots+\lambda^{-n}\tau^n)^2e_{k}.$
A direct calculation shows that $a\ge n\lambda^{-n}e_{k+n},$ whence
\begin{equation}
\label{equa19} 
\|a\|_{E}\ge n\lambda^{-n}s_{k+n}.
\end{equation}

Now, we estimate the norm $\|\mn ^2a\|_{E}$ from above. First,
\begin{multline*}
\mn ^2(I+\lambda^{-1}\tau+\dots +\lambda^{-n}\tau^n)^2=\\
=\lambda(\tau-\lambda I)(\lambda^{-(n+1)}\tau^{n+1}-I)(I+\lambda^{-1}\tau+\dots +\lambda^{-n}\tau^n)=\\
=\lambda^{2}I-2\lambda^{-(n-1)}\tau^{n+1}+\lambda^{-2n}\tau^{2n+2}.
\end{multline*}
Consequently,
$$
\mn ^2a=\lambda^{2}e_{k}-2\lambda^{-(n-1)}e_{n+k+1}+\lambda^{-2n}e_{2n+k+2},$$
and from the triangle inequality it follows
\begin{eqnarray*}
\|\mn ^2a\|_{E}&\le&\lambda^{2}s_{k}+2\lambda^{-(n-1)}s_{n+k+1}+\lambda^{-2n}s_{2n+k+2}\\
&\le& \lambda^2s_{k}+2\lambda^{-(n-1)}\|\tau\|_{E\to E}s_{n+k}+\lambda^{-2n}\|\tau\|_{E\to E}^2s_{2n+k}.
\end{eqnarray*}
Hence,
$$
\|\mn ^2a\|_{E}-2\lambda\|\tau\|_{E\to E}\cdot \lambda^{-n}s_{k+n}\le 2\max(\lambda^{2},\|\tau\|_{E\to E}^2)\max(s_{k},\lambda^{-2n}s_{2n+k}).$$
Let us observe that \eqref{equa18} and \eqref{equa19} yield
$$
\|\mn ^2a\|_{E}\ge c^2n\lambda^{-n}s_{n+k}.$$
Therefore, choosing $n\in\mathbb{N}$ so that 
\begin{equation}
\label{eu8} 
c^2n>2\lambda\|\tau\|_{E\to E}+2\max(\lambda^{2},\|\tau\|_{E\to E}^2),\end{equation}
from the preceding inequality we infer 
$$
\lambda^{-n}s_{n+k}<\max(s_{k},\lambda^{-2n}s_{2n+k}),$$
or, equivalently,
\begin{equation}
\label{equa20} 
\nu _{n+k}<\max(\nu _{k},\nu _{2n+k})\;\;\mbox{for all}\;\;k\in\hj,
\end{equation}
where $\nu_j:=\lambda^{-j}s_j$, $j\in\hj$.

By assumption, $\lambda <2^\delta$. Therefore, for the chosen $n\in\mathbb{N}$ we can find $k\in\hj$ such that $s_{n+k}>\lambda^{n}s_{k}$, i.e., $\nu _{n+k}>\nu _{k}$. Hence,  \eqref{equa20} yields $\nu _{2n+k}>\nu _{n+k}$. 

Substituting $k + n$ for $k$ in \eqref{equa20} and taking into account that $\nu _{2n+k}>\nu _{n+k}$, we obtain $\nu _{3n+k}>\nu _{2n+k}$. Proceeding in the same way, we see that, for the above $n\in\mathbb{N}$ and $k\in\hj$, the sequence $(\nu _{k+rn})_{r=0}^\infty$ is increasing.

Let $j\ge k$ and $m\in\mathbb{N}.$ We find $1\le r_1\le r_2$ satisfying
$$
k+(r_1-1)n \le j \le k+r_1n\;\;\mbox{and}\;\;k+r_2n \le j+m \le k+(r_2+1)n.$$
Denoting $C_1:=\max_{l=1,2,\dots,n}\|\tau_{-l}\|_{E\to E}$, we have
$$
s_j\le C_1s_{k+r_1n}\;\;\mbox{and}\;\;s_{k+r_2n}\le C_1s_{j+m}.$$
Therefore, since $\nu _{k+r_2n}\ge \nu _{k+r_1n}$, 
\begin{equation}
\label{equa20a} 
\frac{s_{j+m}}{s_j}\ge C_1^{-2}\frac{s_{k+r_2n}}{s_{k+r_1n}}=
C_1^{-2}\frac{\lambda^{r_2n}\nu _{k+r_2n}}{\lambda^{r_1n}\nu _{k+r_1n}}\ge C_1^{-2}\lambda^{n(r_2-r_1)}.
\end{equation}

If $\lambda\ge 1$, from the inequality $m-2n\le(r_2-r_1)n$ and \eqref{equa20a} we deduce that
$$
\frac{s_{j+m}}{s_j}\ge C_1^{-2}\lambda^{m-2n}\;\;\mbox{for all}\;j\ge k\;\mbox{and}\;m\in\mathbb{N},$$
or
$$
\sup_{j\ge k}\frac{s_j}{s_{j+m}}\le C_2\lambda^{-m}\;\;\mbox{for all}\;m\in\mathbb{N}.$$
If $k\le 0$, we immediately get  
\begin{equation}
\label{eu7} 
\sup_{j\ge 0}\frac{s_j}{s_{j+m}}\le C_2\lambda^{-m},\;\;m\ge n.
\end{equation}
In the case when $k>0$, denoting $C_3:=\max_{l=\pm 1,\pm 2,\dots,\pm k}\|\tau_{l}\|_{E\to E}$, for all $0\le j<k$ one has
$$
\frac{s_j}{s_{j+m}}\le\frac{s_j}{s_{k}}\cdot\frac{s_{k+m}}{s_{j+m}}\cdot\frac{s_{k}}{s_{k+m}}\le C_3^2\frac{s_k}{s_{k+m}}.$$
Then, combining the last estimates, we come to inequality \eqref{eu7} again (possibly, with a different constant).
Hence, in view of the definition of ${\gamma^\infty}$, we conclude that ${\gamma^\infty}\ge \log_2\lambda$.

If $0<\lambda<1$, then from the inequality $m\ge(r_2-r_1)n$ and estimate \eqref{equa20a} it follows 
$$
\frac{s_{j+m}}{s_j}\ge C_1^{-2}\lambda^{m}\;\;\mbox{for all}\;j\ge k\;\mbox{and}\;m\in\mathbb{N}.$$
Then, reasoning precisely in the same way as in the case when $\lambda\ge 1$, we again get that ${\gamma^\infty}\ge \log_2\lambda$.

To prove the inequality $\log_2\lambda\ge {\delta^0}$, we will use the hypothesis that $\log_2\lambda>\gamma$. 
For $n$ satisfying \eqref{eu8} we can find $k\in\hj$ such that $s_{n+k}>\lambda^{-n}s_{2n+k}$, i.e., $\nu _{n+k}>\nu _{2n+k}$. Therefore,  by \eqref{equa20}, $\nu _{n+k}<\nu _{k}$. 
Substituting $k -n$ for $k$ in \eqref{equa20}, we obtain $\nu _{k}<\nu _{k-n}$. Proceeding in the same way, we see that for the chosen $n\in\mathbb{N}$ and $k\in\hj$ the sequence $(\nu _{k-rn})_{r=0}^\infty$  is increasing. 

If $j\le k$ and $m\in\mathbb{N},$ then there are $1\le r_1\le r_2$, for which
$$
k-r_1n\le j\le k-(r_1-1)n\;\;\mbox{and}\;\;k-(r_2+1)n\le j-m \le k-r_2n.$$
Setting $C_4:=\max_{l=1,2,\dots,n}\|\tau_{l}\|_{E\to E}$, we have
$$
s_{j}\le C_4s_{k-r_1n}\;\;\mbox{and}\;\;s_{k-r_2n}\le C_4s_{j-m}.$$
Therefore, since $\nu _{k-r_2n}\ge \nu _{k-r_1n}$, 
\begin{equation}
\label{eu10} 
\frac{s_{j}}{s_{j-m}}\le C_4^{2}\frac{s_{k-r_1n}}{s_{k-r_2n}}=
C_4^{2}\frac{\lambda^{-r_1n}\nu_{k-r_1n}}{\lambda^{-r_2n}\nu_{k-r_2n}}\le C_4^{2}\lambda^{n(r_2-r_1)}\;\;\mbox{for all}\;j\le k\;\mbox{and}\;m\in\mathbb{N}.
\end{equation}
If $\lambda\ge 1$, from the inequality $(r_2-r_1)n\le m$ and \eqref{eu10} it follows
$$
\sup_{j\le k}\frac{s_{j}}{s_{j-m}}\le C_4^{2}\lambda^{m}\;\;\mbox{for all}\;m\in\mathbb{N}.$$
Consequently, if $k\ge 0$, we obtain 
\begin{equation*}
\sup_{j\le 0}\frac{s_{j}}{s_{j-m}}\le C_4^2\lambda^{m},\;\;m\in\mathbb{N}.
\end{equation*}
Arguing as above, one can check that the same inequality holds (maybe, with a different constant) also in the case when $k<0$. From this estimate and  the definition of ${\delta^0}$ it follows immediately that ${\delta^0}\le \log_2\lambda$.

In the case when $0<\lambda<1$, combining the inequality $m\le(r_2-r_1+2)n$ with \eqref{eu10}, we get 
$$
\frac{s_{j}}{s_{j-m}}\le C_4^{2}\lambda^{m-2n}\;\;\mbox{for all}\;j\le k\;\mbox{and}\;m\in\mathbb{N},$$
whence 
$$
\sup_{j\le k}\frac{s_{j}}{s_{j-m}}\le C_5\lambda^{m},\;\;m\in\mathbb{N}.$$
The desired inequality ${\delta^0}\le \log_2\lambda$ can be obtained now by repeating the same reasoning as above.
As a result, implication \eqref{eu21a} is proved.

In conclusion, we need to establish that the numbers $2^\gamma$, $2^\delta$, $2^{{\delta^0}}$ and $2^{{\gamma^\infty}}$ are approximative eigenvalues of the operator $\tau$. For $2^\gamma$ and $2^\delta$ it can be proved precisely in the same way as in \cite[Theorem~6]{A-11} (see also \cite[Lemma~11.3.12]{AK}). Hence, we may assume that $\gamma<\delta^0\le \gamma^\infty<\delta$. In the case when $\delta^0< \gamma^\infty$, the above proof shows that $\tau_\lambda$ is a Fredholm operator whenever $\lambda\in(2^{{\delta^0}},2^{{\gamma^\infty}})$. On the other hand, it is well known (see e.g. \cite[Theorem~III.21]{KG}) that  Fredholm operators form an open set in the space of all bounded linear operators on ${E}$ with respect to the topology generated by the operator norm. Therefore, from the fact that $\tau_\lambda$ fails to be isomorphic for $2^\gamma<\lambda< 2^{{\delta^0}}$ and $2^{{\gamma^\infty}}<\lambda< 2^\delta$ it follows that $\tau_\lambda$ is not isomorphic if $\lambda$ equals $2^{\delta^0}$ or $2^{\gamma^\infty}$. Since the same argument can be applied also in the case when $\delta^0=\gamma^\infty$, the theorem is proved.
\end{proof}

\smallskip

Now, we are ready to prove Theorem \ref{Theorem 4} which gives a description of the set of approximative values of the doubling operator $\sigma$ in a separable r.i. function space $X$ of fundamental type.

\begin{proof}[Proof of Theorem \ref{Theorem 4}]
It suffices to combine the results of Lemma \ref{dilation1}, Theorem \ref{L2-new} for $E=E_X$ and Proposition \ref{Proposition 1}.
\end{proof}

As a consequence, we get Corollary \ref{main1}, showing that the set of $p$ such that $\ell^p$ is symmetrically finitely represented in a separable r.i. function space $X$ of fundamental type either an interval or a union of two intervals. 

\smallskip
For r.i. spaces $X$ of a non-fundamental type, the set ${\mathcal F}(X)$ might be apparently more complicated. Nevertheless, some results can be obtained also in the general case.

Let $X$ be a separable r.i. space on $(0,\infty)$ and let $\phi_X$ be the fundamental function of $X$. Then, by the definition of $E_X$, the dilation indices of the function $M_{\phi_X}$ (see Subsection \ref{prel3}) can be calculated as follows:
\begin{align*}
\mu_{\phi_X} &= -\lim_{n\to\infty}\frac{1}{n}\log_2
\sup_{k\in\mathbb{Z}}\frac{s _k}{s _{n+k}}, 
&\nu_{\phi_X}&=\lim_{n\to\infty}\frac{1}{n}\log_2 \sup_{k\in\mathbb{Z}}
\frac{s_{k+n}}{s _{k}},
\nonumber\\
{\mu_{\phi_X}^0}&= -\lim_{n\to\infty}\frac{1}{n}\log_2 \sup_{k\le 0}\frac{s_{k-n}}{s_{k}}, 
&\nu_{\phi_X}^0&=\lim_{n\to\infty}\frac{1}{n}\log_2 \sup_{k\le 0}\frac{s_k}{s_{k-n}},
\nonumber\\ \qquad
\mu_{\phi_X}^\infty&= -\lim_{n\to\infty}\frac{1}{n}\log_2 \sup_{k\ge 0}\frac{s_k}{s_{n+k}}, 
&{\nu_{\phi_X}^\infty} &=\lim_{n\to\infty}\frac{1}{n}\log_2 \sup_{k\ge 0}\frac{s_{k+n}}{s_k},
\end{align*}
where $s_j=\|e_j\|_{E_X}=\|\chi_{\Delta_j}\|_X$, $j\in\hj$.
Therefore, repeating the second part of the proof of  Theorem \ref{L2-new} for $E=E_X$ and applying Proposition \ref{Proposition 1}, we get the following result. 

\begin{theor}\label{Theorem 4f}
Suppose $X$ is a separable r.i. space on $(0,\infty)$. If ${\mu_{\phi_X}^\infty}\le {{\nu_{\phi_X}^0}}$ (resp. ${\mu_{\phi_X}^\infty}>{{\nu_{\phi_X}^0}}$), then the interval $[2^{\mu_{\phi_X}},2^{\nu_{\phi_X}}]$ (resp. the union  $[2^{\mu_{\phi_X}},2^{\nu_{\phi_X}^0}] \cup [2^{\mu_{\phi_X}^\infty},2^{\nu_{\phi_X}}]$) consists of  approximate eigenvalues of the operator $\sigma$ in $X$.
\end{theor}

Let $\lambda$ be an approximate eigenvalue of the operator $\sigma$ in a separable r.i. space $X$. Then, by Theorem \ref{Theorem 1}, $p=(\log_2\lambda)^{-1}\in {\mathcal F}(X)$. One can easily see that for every $\epsilon>0$ and $n\in\mathbb{N}$ we can take, as pairwise disjoint and equimeasurable functions in $X$ that correspond to the unit basis vectors of $\ell^p$ (see inequalities \eqref{eq0}), linear combinations of characteristic functions of dyadic intervals, i.e., a finite block basis of the sequence $\{\chi_{\Delta_{n}^k}\}_{n,k}$, $\Delta_n^k=[2^{-n}(k-1), 2^{-n}k)$, $n=0,1,\dots$, $k=1,2,\dots$. This observation combined with Theorem \ref{Theorem 4f} implies the following result, which shows that in the case of r.i. spaces the sequence of characteristic functions of dyadic intervals plays, roughly speaking, the same role as a subsymmetric unconditional basis in spaces with such a basis (it should be compared with the Rosenthal version of Krivine's theorem \cite[p.~198]{Ros}; see Theorem \ref{Th:Rosenthal} in Section \ref{Intro}).

\begin{theor}\label{Theorem 5f}
Let $X$ be a separable r.i. space on $(0,\infty)$ and let $\phi_X$ be the fundamental function of $X$. If ${\mu_{\phi_X}^\infty}\le {{\nu_{\phi_X}^0}}$ (resp. ${\mu_{\phi_X}^\infty}>{{\nu_{\phi_X}^0}}$), then $\ell^p$ is symmetrically finitely represented in the sequence $\{\chi_{\Delta_{n}^k}\}_{n,k}$ whenever $p\in [1/\nu_{\phi_X},1/\mu_{\phi_X}]$ (resp. $p\in [1/\nu_{\phi_X},1/\mu_{\phi_X}^\infty]\cup [1/\nu_{\phi_X}^0,1/\mu_{\phi_X}]$).
\end{theor}

\smallskip

Results obtained allow us to give a description of the set ${\mathcal F}({X})$ in the case when $X$ is a Lorentz or a separable Orlicz sequence space.

\vskip0.5cm

\section{A description of the set of $p$ such that $\ell^p$ is symmetrically finitely represented in Lorentz and Orlicz spaces.}

\subsection{Lorentz spaces.}

Let $1\le q<\infty$, and let $\psi$ be an increasing concave function on $[0,\infty)$ such that $\psi(0)=0$. Recall that the Lorentz space $\Lambda_q(\psi)$ is endowed with the norm
\begin{equation*}
\label{eqLor} 
\|x\|_{\Lambda_q(\psi)}=\left(\int_0^\infty [x^{*}(t)]^{q}\, 
d\psi(t) \right)^{1/q}.
\end{equation*} 
It is easy to check (see also \cite[\S\,II.4.4]{KPS} or \cite[p.~28]{Ma-85}) that $\Lambda_q(\psi)$ is a r.i. space of fundamental type. Therefore, since $\phi_{\Lambda_q(\psi)}=\psi^{1/q}$, the dilation indices of this space can be calculated by formulae \eqref{Ind of Lor}, i.e.,
\begin{align*}
\alpha_{\psi,q} &= -\lim_{n\to\infty}\frac{1}{n}\log_2
\sup_{k\in\mathbb{Z}}\Big(\frac{\psi(2^k)}{\psi(2^{n+k})}\Big)^{1/q}, 
&\beta_{\psi,q} &=\lim_{n\to\infty}\frac{1}{n}\log_2 \sup_{k\in\mathbb{Z}}
\Big(\frac{\psi(2^{n+k})}{\psi(2^k)}\Big)^{1/q},
\nonumber\\
\alpha_{\psi,q}^0&= -\lim_{n\to\infty}\frac{1}{n}\log_2 \sup_{k\le 0}\Big(\frac{\psi(2^{k-n})}{\psi(2^k)}\Big)^{1/q}, 
&\beta_{\psi,q}^0&=\lim_{n\to\infty}\frac{1}{n}\log_2 \sup_{k\le 0}\Big(\frac{\psi(2^{k})}{\psi(2^{k-n})}\Big)^{1/q},
\nonumber\\ \qquad
\alpha_{\psi,q}^\infty&= -\lim_{n\to\infty}\frac{1}{n}\log_2 \sup_{k\ge 0}\Big(\frac{\psi(2^{k})}{\psi(2^{k+n})}\Big)^{1/q}, 
&\beta_{\psi,q}^\infty &=\lim_{n\to\infty}\frac{1}{n}\log_2 \sup_{k\ge 0}\Big(\frac{\psi(2^{n+k})}{\psi(2^k)}\Big)^{1/q}.
\end{align*}
Taking into account that $\Lambda_q(\psi)$ is separable, as an immediate consequence of Corollary \ref{main1}, we get the following description of the set of all $p$ such that $\ell^p$ is symmetrically finitely represented in Lorentz spaces.

\begin{theor}\label{Theorem 4a}
Let $1\le q<\infty$, and let $\psi$ be an increasing concave function on $[0,\infty)$ such that $\psi(0)=0$. Then, we have:

(i) if ${\alpha_{\psi,q}^\infty}\le {{\beta_{\psi,q}^0}}$, then ${\mathcal F}(\Lambda_q(\psi))=[1/\beta_{\psi,q},1/\alpha_{\psi,q}]$;

(ii) if ${\alpha_{\psi,q}^\infty}>{{\beta_{\psi,q}^0}}$, then ${\mathcal F}(\Lambda_q(\psi))=[1/\beta_{\psi,q},1/\alpha_{\psi,q}^\infty]\cup [1/\beta_{\psi,q}^0,1/\alpha_{\psi,q}]$.
\end{theor}

Assuming that $\mu_\psi>0$, we show that the Banach sequence lattice $E_{\Lambda_q(\psi)}$ coincides (up to equivalence of norms) with the weighted $\ell^q$-space equipped with the following natural norm:
$$
\|a\|_{\ell^q(\psi)}=\Big(\sum_{k\in\hj} |a_{k}|^q\psi(2^{k})\Big)^{1/q}.$$
 
Indeed, from the concavity of $\psi$ it follows that  
$$
\|x\|_{\Lambda_q(\psi)}^q\le \sum_{k\in\hj}\int_{2^k}^{2^{k+1}} [x^{*}(t)]^{q}\,d\psi(t)\le 2\sum_{k\in\hj}[x^{*}(2^k)]^{q}\psi(2^k).$$
On the other hand, since $M_\psi(1/2)< 1$ (see e.g. \cite[Lemma~II.1.4]{KPS}), we have
$$
\|x\|_{\Lambda_q(\psi)}^q\ge \sum_{k\in\hj}[x^{*}(2^{k+1})]^{q}\psi(2^{k+1})\Big(1-\frac{\psi(2^{k})}{\psi(2^{k+1})}\Big)\ge (1-M_\psi(1/2))\sum_{k\in\hj}[x^{*}(2^{k})]^{q}\psi(2^{k}).$$

Therefore,
\begin{equation*}
\label{equa21} 
\|x\|_{\Lambda_q(\psi)}\asymp \Big\|\sum_{k\in\hj} x^{*}(2^{k})e_{k}\Big\|_{\ell^q(\psi)}.
\end{equation*}
Since $\gamma_{\ell^q(\psi)}=\mu_{\psi^{1/q}}=\mu_\psi/q>0$, by Proposition \ref{prop:latttice and symm}, we obtain that $E_{\Lambda_q(\psi)}=\ell^q(\psi)$. 

\smallskip

\subsection{Orlicz spaces.}

Let $N$ be an Orlicz function. Then the Orlicz function space $L_N$ on $(0,\infty)$ is equipped with the norm 
$$
\|x\|_{L_N}=\inf\Big\{u
>0\,:\,\int_0^\infty N(|x(t)|/u) \,dt\leq 1\Big\}.$$ 
Since every Orlicz space is of fundamental type (see e.g. \cite{Boyd} or \cite[Theorem~4.2]{Ma-85}) and $\phi_{L_N}(t)=1/N^{-1}(1/t)$, $t>0$, where $N^{-1}$ is the inverse function for $N$, the dilation indices of the space $L_N$ can be calculated by the formulae:
\begin{align}
\alpha_{N} &= -\lim_{n\to\infty}\frac{1}{n}\log_2
\sup_{k\in\mathbb{Z}}\frac{N^{-1}(2^{k-n})}{N^{-1}(2^{k})}, 
&\beta_{N} &=\lim_{n\to\infty}\frac{1}{n}\log_2 \sup_{k\in\mathbb{Z}}
\frac{N^{-1}(2^{k})}{N^{-1}(2^{k-n})},
\nonumber\\
\alpha_{N}^0&= -\lim_{n\to\infty}\frac{1}{n}\log_2 \sup_{k\ge 0}\frac{N^{-1}(2^{k})}{N^{-1}(2^{k+n})}, 
&\beta_{N}^0&=\lim_{n\to\infty}\frac{1}{n}\log_2 \sup_{k\ge 0}\frac{N^{-1}(2^{k+n})}{N^{-1}(2^{k})},
\nonumber\\ \qquad
\alpha_{N}^\infty&= -\lim_{n\to\infty}\frac{1}{n}\log_2 \sup_{k\le 0}\frac{N^{-1}(2^{k-n})}{N^{-1}(2^{k})}, 
&\beta_{N}^\infty &=\lim_{n\to\infty}\frac{1}{n}\log_2 \sup_{k\le 0}\frac{N^{-1}(2^{k})}{N^{-1}(2^{k-n})}.
\label{eq6b}
\end{align}
Recall that an Orlicz space $L_N$ is separable if and only if the function $N$ satisfies the $\Delta_2$-condition (see \cite[\S\,II.10]{KR} or Subsection \ref{prel2}). One can easily check that the latter condition is equivalent to the fact that $\alpha_{N}>0$. Thus, by Corollary \ref{main1}, we get the following result.

\begin{theor}\label{Theorem 4b}
If an Orlicz function $N$ is such that $\alpha_{N}>0$ (see \eqref{eq6b}), then 

(i) if ${\alpha_{N}^\infty}\le {{\beta_{N}^0}}$, then ${\mathcal F}(L_N)=[1/\beta_{N},1/\alpha_{N}]$;

(ii) if ${\alpha_{N}^\infty}>{{\beta_{N}^0}}$, then ${\mathcal F}(L_N)=[1/\beta_{N},1/\alpha_{N}^\infty]\cup [1/\beta_{N}^0,1/\alpha_{N}]$.
\end{theor}

Let $L_N$ be a separable Orlicz space. Show that the space $E_{L_N}$ coincides (with equivalence of norms) with the Banach sequence lattice $U_N$ equipped with the norm
$$
\|a\|_{U_N}:=\inf\left\{u>0:\,\sum_{k\in\hj} 2^{k} N\Big(\frac{|a_k|}{u}\Big)\le 1\right\}.$$

To this end, we check first that
\begin{equation}
\label{equa261} 
\frac12 \|(x^*(2^{k}))_{k\in\hj}\|_{U_N}
\le \|x\|_{L_N}\le \|(x^*(2^{k}))_{k\in\hj}\|_{U_N}.
\end{equation}

Since $L_N$ is a r.i. space, we can assume that $x=x^*$. By using the monotonicity of the function $N$, we have
\begin{equation}
\label{equa27} 
\int_{0}^\infty N\Big(\frac{x(t)}{u}\Big)\,dt=\sum_{k\in\hj}\int_{2^{k}}^{2^{k+1}} N\Big(\frac{x(t)}{u}\Big)\,dt\le\sum_{k\in\hj} 2^{k} N\Big(\frac{x(2^{k})}{u}\Big),
\end{equation}
and similarly in the opposite direction
\begin{equation}
\label{equa28} 
\int_{0}^\infty N\Big(\frac{x(t)}{u}\Big)\,dt\ge \sum_{k\in\hj} 2^{k} N\Big(\frac{x(2^{k+1})}{u}\Big)=\frac12\sum_{k\in\hj} 2^{k} N\Big(\frac{x(2^{k})}{u}\Big).
\end{equation}
Next, assuming that $u > \|(x(2^{k}))\|_{U_N}$, by the definition of the $U_N$-norm and inequality \eqref{equa27}, we infer
$$
\int_{0}^\infty N\Big(\frac{x(t)}{u}\Big)\,dt\le\sum_{k\in\hj} 2^{k} N\Big(\frac{x(2^{k})}{u}\Big)\le 1,$$
which implies that $u\ge \|x\|_{L_N}$. Hence, $\|x\|_{L_N}\le\|(x(2^{k}))\|_{U_N}$.

Conversely, if $u >\|x\|_{L_N}$, then from \eqref{equa28} and the convexity of $N$ it follows
$$
\sum_{k\in\hj} 2^{k} N\Big(\frac{x(2^{k})}{2u}\Big)
\le \frac12\sum_{k\in\hj} 2^{k} N\Big(\frac{x(2^{k})}{u}\Big)\le\int_{0}^\infty N\Big(\frac{x(t)}{u}\Big)\,dt\le 1,$$
whence $2u\ge \|(x(2^{k}))\|_{U_N}$. Consequently, $\|(x(2^{k}))\|_{U_N}\le 2\|x\|_{L_N}$, and so \eqref{equa261} is proved.

Let us prove now that $\gamma_{U_N}=\alpha_{N}$, or equivalently (see
\eqref{eq6b}) that
\begin{equation}
\label{equa281} 
\gamma_{U_N}=-\lim_{n\to\infty}\frac{1}{n}\log_2
\sup_{k\in\mathbb{Z}}\frac{N^{-1}(2^{k-n})}{N^{-1}(2^{k})}.
\end{equation}

Observe that for any $s>0$ and $n\in\mathbb{N}$ we have
\begin{equation}
\label{equa30} 
\frac{N^{-1}(2^{-n}s)}{N^{-1}(s)}\le 2A_n,
\end{equation}
where
$$
A_n:=\sup_{m\in\hj}\frac{N^{-1}(2^{m-n})}{N^{-1}(2^{m})},\;\;n\in\mathbb{N}.$$
Indeed, choosing an integer $m$ so that $s\in (2^{m},2^{m+1}]$, since the inverse function $N^{-1}$ is nondecreasing and concave, we get
$$
\frac{N^{-1}(2^{-n}s)}{N^{-1}(s)}\le \frac{N^{-1}(2^{m-n+1})}{N^{-1}(2^{m})}\le 2\frac{N^{-1}(2^{m-n})}{N^{-1}(2^{m})}\le 2A_n.$$
Moreover, one can readily check that \eqref{equa30} may be rewritten as follows:
\begin{equation}
\label{equa31} 
N\Big(\frac{t}{2A_n}\Big)\le 2^{n}N(t),\;\;t>0.
\end{equation}

Let $n\in\mathbb{N}$, $a=(a_k)_{k\in\hj}\in U_N$ and $u>\|a\|_{U_N}$. Then, by the definition of the $U_N$-norm and \eqref{equa31}, we have
\begin{eqnarray*}
\sum_{k\in\hj} 2^{k} N\Big(\frac{|(\tau_{-n}a)_{k}|}{2A_nu}\Big)&=&\sum_{k\in\hj} 2^{k}N\Big(\frac{|a_{k+n}|}{2A_nu}\Big)\le \sum_{k\in\hj} 2^{k+n} N\Big(\frac{|a_{k+n}|}{u}\Big)\\
&=&\sum_{k\in\hj} 2^{k} N\Big(\frac{|a_{k}|}{u}\Big)\le 1,
\end{eqnarray*}
which implies that $\|\tau_{-n}a\|_{U_N}\le 2A_n$. Therefore, from the 
definition of $A_n$ we infer that 
$$
\|\tau_{-n}\|_{U_N\to U_N}\le 2\sup_{m\in\hj}\frac{N^{-1}(2^{m-n})}{N^{-1}(2^{m})},\;\;n\in\mathbb{N}.$$
Since $\|e_j\|_{U_N}=1/N^{-1}(2^{-j})$ for all $j\in\hj$, the opposite inequality
$$
\|\tau_{-n}\|_{U_N\to U_N}\ge \sup_{m\in\hj}\frac{N^{-1}(2^{m-n})}{N^{-1}(2^{m})},\;\;n\in\mathbb{N},$$
is obvious. Therefore, in view of the definition of $\gamma_{U_N}$, we get \eqref{equa281}, and hence $\gamma_{U_N}=\alpha_{N}$.

Since $L_N$ is separable, $\gamma_{U_N}$ is positive together with $\alpha_{N}$. Combining this with inequalities \eqref{equa261}, by  Proposition \ref{prop:latttice and symm}, we infer that $E_{L_N}=U_N$. 

\begin{remark}
Similar results hold also for separable parts of non-separable Orlicz spaces (see the definition in Subsection \ref{prel2}).
\end{remark}

\vskip0.5cm

\section{Appendix: Closure of the operator $\tau_\lambda$ in Banach sequence lattices}

Here, we prove a complement to Theorem \ref{L2-new}, where for brevity we set $\gamma:=\gamma_E$, $\gamma^0:=\gamma_E^0$,  $\gamma^\infty:=\gamma_E^\infty$, $\delta:=\delta_E$, $\delta^0:=\delta_E^0$, $\delta^\infty:=\delta_E^\infty$.

\begin{theor}\label{L2-new-close}
Suppose a separable Banach sequence lattice $E$ satisfies condition \eqref{equa16a}. Then, we have

(a) if ${\gamma^\infty}\le {{\delta^0}}$ and ${\gamma^0}\le {{\delta^\infty}}$, then the operator $\mn$ is closed in ${E}$ if and only if  $\lambda\in (0,2^\gamma)\cup (2^{\delta},\infty)$;

(b) if ${\gamma^\infty}>{{\delta^0}}$, then $\mn$ is closed in ${E}$ if and only if $\lambda\in (0,2^\gamma)\cup (2^{{\delta^0}},2^{{\gamma^\infty}})\cup (2^{\delta},\infty)$;

(c) if ${\gamma^0}>{{\delta^\infty}}$, then $\mn$ is closed in ${E}$ if and only if $\lambda\in (0,2^\gamma)\cup (2^{{\delta^\infty}},2^{{\gamma^0}})\cup (2^{\delta},\infty)$.

In the cases (a) and (b) if $\mn$ is closed in ${E}$, then it is both an  isomorphism in $E$, while in the case (c) ${\rm Im}\,\mn=E$, but $\mn$ is not injective.
\end{theor}

\begin{remark}
In view of inequalities \eqref{eq19.10}, at most only one of the intervals $(2^{{\delta^0}},2^{{\gamma^\infty}})$ and $(2^{{\delta^\infty}},2^{{\gamma^0}})$ may be non-empty.
\end{remark}

\begin{proof}
First, suppose ${\delta^\infty}<{\gamma^0}$ and $\lambda\in (2^{{\delta^\infty}},2^{{\gamma^0}})$. We show that the operator $\mn$ is a surjective and not injective mapping in $E$.

Let $\varepsilon>0$ satisfy the inequality  
\begin{equation}
\label{eu27} 
{\delta^\infty}+\varepsilon<\log_2\lambda< {\gamma^0}-\varepsilon.
\end{equation}
Then, for some $C>0$ 
\begin{equation}
\label{eu28} 
\sup_{k\le 0}\frac{s_{k-n}}{s_k}\le C2^{-({\gamma^0}-\varepsilon) n}\;\;\mbox{and}\;\;\sup_{k\ge 0}\frac{s_{n+k}}{s_k}\le C2^{({\delta^\infty}+\varepsilon) n},\;\;n\in\mathbb{N}.
\end{equation}

In particular, from \eqref{eu28} it follows that
\begin{equation*}
s_{-n}\le Cs_{0}2^{-({\gamma^0}-\varepsilon) n}\;\;\mbox{and}\;\;s_{n}\le Cs_{0}2^{({\delta^\infty}+\varepsilon) n},\;\;n=0,1,\dots.
\end{equation*}
Therefore, by the triangle inequality,
$$
\Big\|\sum_{n\in\hj} \lambda^{-n}e_n\Big\|_E\le\sum_{n\in\hj} s_n\lambda^{-n}\le Cs_0\Big(\sum_{n=0}^\infty (\lambda^{-1}2^{\delta^\infty+\varepsilon})^n+ \sum_{n=1}^\infty (\lambda 2^{-({\gamma^0}-\varepsilon)})^{n}\Big)<\infty,$$
whence $a^0:=(\lambda^{-n})_{n\in\hj}\in E$. Since $\tau_\lambda a^0=0$ (see the proof of Theorem \ref{L2-new}), this implies that the operator $\mn$ is not injective.

To prove that ${\rm Im}\,\mn=E$, we consider the operator $T$ defined by
$$
Tx:=-\sum_{n=1}^\infty\sum_{i=1}^n \lambda^{i-n-1}x_i e_n+\sum_{n=1}^\infty\sum_{i=1}^n \lambda^{n-i}x_{-i+1} e_{-n},\;\;\mbox{where}\;x=(x_i)_{i\in\hj}.$$
Let us show that $T$ is bounded on $E$. To this end, recalling that $E'$ is the K\"{o}the dual for $E$ (see Subsection \ref{prel1}), for each $b=(b_i)_{i\in\hj}\in E'$, $\|b\|_{E'}\le 1$, we estimate:
\begin{eqnarray*}
\sum_{n\in\hj} b_n(Tx)_n &=& -\sum_{n=1}^\infty b_n\sum_{i=1}^n \lambda^{i-n-1}x_i+\sum_{n=1}^\infty b_{-n}\sum_{i=1}^n \lambda^{n-i}x_{-i+1}\\ &=&
-\sum_{n=1}^\infty b_n\sum_{k=1}^n \lambda^{-k}x_{n-k+1}+\sum_{n=1}^\infty b_{-n}\sum_{k=1}^n \lambda^{k-1}x_{k-n}\\ &=&
-\sum_{k=1}^\infty \lambda^{-k} \sum_{n=k}^\infty b_nx_{n-k+1}+\sum_{k=1}^\infty \lambda^{k-1}\sum_{n=k}^\infty b_{-n}x_{k-n}\\ &\le&
\sum_{k=1}^\infty \lambda^{-k} \left\|(b_n)_{n=k}^\infty\right\|_{E'}
\Big\|\sum_{n=k}^\infty x_{n-k+1}e_n\Big\|_E\\ &+& \sum_{k=1}^\infty \lambda^{k-1} \left\|(b_{-n})_{n=k}^\infty\right\|_{E'}
\Big\|\sum_{n=k}^\infty x_{k-n}e_{-n}\Big\|_E\\ &\le&
\sum_{k=1}^\infty \lambda^{-k}\left\|\tau_k((x_n)_{n=1}^\infty)\right\|_{E}+\sum_{k=1}^\infty \lambda^{k-1}\left\|\tau_{-k-1}((x_{-n})_{n=0}^\infty)\right\|_{E}.
\end{eqnarray*}
According to \eqref{equa16a}, it holds
$$
\left\|\tau_k((x_n)_{n=1}^\infty)\right\|_{E}\le C\sup_{j\ge 0}\frac{s_{k+j}}{s_j}\;\;\mbox{and}\;\;\left\|\tau_{-k-1}((x_{-n})_{n=0}^\infty)\right\|_{E}\le C\sup_{j\le 0}\frac{s_{j-k-1}}{s_j},\;\;k\in\mathbb{N}.$$
Furthermore, since $E$ is separable, we have $E'=E^*$. Thus, assuming that $\varepsilon>0$ satisfies \eqref{eu27}, we get
\begin{eqnarray*}
\|Tx\|_E&=& \sup\Big\{\sum_{n\in\hj} b_n(Tx)_n:\,\|(b_n)\|_{E'}\le 1\Big\}\\ &\le& C\sum_{k=1}^\infty \left(\lambda^{-k}2^{({\delta^\infty}+\varepsilon)k}+\lambda^{k-1}2^{({\gamma^0}-\varepsilon)(k-1)}\right)\|x\|_E=C'\|x\|_E,
\end{eqnarray*}
and our claim is proved.

Moreover, one can readily check that $\tau_\lambda(Ta)=a$ for each  $a=(a_i)_{i\in\hj}\in E$, whence $\op=E$ if $\lambda\in (2^{{\delta^\infty}},2^{{\gamma^0}})$.

It remains to prove that $\lambda\in (0,2^\gamma)\cup (2^{{\delta^0}},2^{{\gamma^\infty}})\cup(2^{{\delta^\infty}},2^{{\gamma^0}})\cup (2^{\delta},\infty)$ whenever $\mn$ is closed in $E$. If additionally $\mn$ is injective then, by Theorem \ref{L2-new}, $\lambda\in (0,2^\gamma)\cup (2^{{\delta^0}},2^{{\gamma^\infty}})\cup (2^{\delta},\infty)$. Therefore, we need only to prove the following implication: if $\mn$ is closed and not injective, then  $\lambda\in(2^{{\delta^\infty}},2^{{\gamma^0}})$.

If the operator $\mn$ is not injective, then $(\lambda^{-n})_{n\in\hj}\in E,$ and so the sequence $(\lambda^{n})_{n\in\hj}$ defines an unbounded functional on $E$. Therefore, since $\mn$ is closed, by the proof of Theorem \ref{L2-new} (see \eqref{eq16c}), $\op=E$. Hence, in view of the well-known duality result \cite[B.3.9, Proposition~2]{Pietsch},  the dual operator ${\tau_\lambda^*}=\tau_{-1}-\lambda I$ is an isomorphism in the space $E^*=E '$. Then, for some $c>0$ and for all $y\in E'$ we have
\begin{equation}
\label{eu20} 
\|{\tau_\lambda^*} y\|_{E'}\ge c\|y\|_{E'}.
\end{equation}

We prove the implication:
\begin{equation}
\label{eu21} 
\lambda\not\in (0,2^\gamma]\cup [2^{\delta},\infty)\;\;\mbox{and inequality}\;\;\eqref{eu20}\;\mbox{holds}\;\;\Longrightarrow\;\; \lambda\in [2^{{\delta^\infty}},2^{{\gamma^0}}].
\end{equation}

For arbitrary  $n\in\fg$ and $k\in\hj$ (which will be chosen later) we set $b:=(I+\lambda^{-1}\tau^{-1}+\dots+\lambda^{-n}\tau^{-n})^2e_{k+n}.$
A direct calculation indicates that $b\ge n\lambda^{-n}e_{k}.$ Therefore, since $\|e_j\|_{E'}=1/\|e_j\|_{E}=s_j^{-1}$, $j\in \hj$, we have
\begin{equation}
\label{eu22} 
\|b\|_{E'}\ge n\lambda^{-n}s_ {k}^{-1}.
\end{equation}
Moreover,
$$
({\tau_\lambda^*}) ^2(I+\lambda^{-1}\tau^{-1}+\dots+\lambda^{-n}\tau^{-n})^2=[{\tau_\lambda^*}
(I+\lambda^{-1}\tau^{-1}+\dots+\lambda^{-n}\tau^{-n})]^2=$$
$$
=\lambda^{2}(\lambda^{-(n+1)}\tau^{-n-1}-I)^2=
\lambda^{2}I-2\lambda^{1-n}\tau^{-n-1}+\lambda^{-2n}\tau^{-2n-2}.$$
Consequently,
$$
({\tau_\lambda^*})^2b=\lambda^{2}e_{k+n}-2\lambda^{1-n}e_{k-1}+\lambda^{-2n}e_{k-n-2},$$
and from the triangle inequality it follows
\begin{eqnarray*}
\|({\tau_\lambda^*})^2b\|_{E'}&\le &\lambda^{2}s_{k+n}^{-1}+2\lambda^{1-n}s_{k-1}^{-1}+\lambda^{-2n}s_{k-n-2}^{-1}\\
&\le& \lambda^2s_{k+n}^{-1}+2\lambda^{1-n}\|\tau\|_{E\to E}s_{k}^{-1}+\lambda^{-2n}\|\tau\|_{E\to E}^2s_{k-n}^{-1}.
\end{eqnarray*}
Hence,
$$
\|({\tau_\lambda^*})^2b\|_{E'}-2\lambda^{1-n}\|\tau\|_{E\to E}s_{k}^{-1}\le 2\max(\lambda^2,\|\tau\|_{E\to E}^2)\max(s_{k+n}^{-1},\lambda^{-2n}s_{k-n}^{-1}).$$
From inequalities \eqref{eu20} and \eqref{eu22} it follows
$$
\|({\tau_\lambda^*})^2b\|_{E'}\ge c^2n\lambda^{-n}s_{k}^{-1}.$$
Therefore, if $n\in{\mathbb N}$ satisfies 
\begin{equation}
\label{eu23} 
c^2n>2\lambda\|\tau\|_{E\to E}+2\max(\lambda^{2},\|\tau\|_{E\to E}^2),\end{equation}
the preceding inequality yields
$$
\lambda^{-n}s_{k}^{-1}<\max(s_{k+n}^{-1},\lambda^{-2n}s_{k-n}^{-1}),$$
or, equivalently, 
\begin{equation}
\label{equa25} 
\nu _{k}<\max(\nu _{k+n},\nu _{k-n}),
\end{equation}
where $\nu_j:=\lambda^{j}s_j^{-1}$, $j\in\hj$.

Since $\log_2\lambda>\gamma,$ for a fixed $n\in {\mathbb N}$, which satisfies \eqref{eu23}, there is $k\in\hj$ such that $s _{k-n}\,>\,\lambda^{-n}s _k.$ Then, $\nu _{k-n}<\nu _k,$ and from \eqref{equa25} it follows that
$\nu_k<\nu _{k+n}.$ Similarly, replacing $k$ in \eqref{equa25} with $k+n$, we get that $\nu _{k+n}<\nu _{k+2n}$. Repeating in the same way, we conclude that the sequence $(\nu _{k+rn})_{r=0}^\infty$ is increasing.

Given $j\ge k$ and $m\ge 0$, we can find $1\le r_1< r_2$ such that
$$
k+(r_1-1)n\le j\le k+r_1n\;\;\mbox{and}\;\;k+r_2n\le j+m\le k+(r_2+1)n.$$
Then, if $C_1:=\max_{l=1,2,\dots,n}\|\tau_{l}\|_{E\to E}$, we have
$$
s_{k+r_1n}\le C_1s_j\;\;\mbox{and}\;\;s_{j+m}\le C_1s_{k+r_2n}.$$
Hence,
\begin{equation}
\label{equa26} 
\frac{s_{j+m}}{s_j}\le C_1^{2}\frac{s_{k+r_2n}}{s_{k+r_1n}}=
C_1^{2}\frac{\lambda^{r_2n}\nu _{k+r_2n}^{-1}}{\lambda^{r_1n}\nu _{k+r_1n}^{-1}}\le C_1^{2}\lambda^{n(r_2-r_1)}.
\end{equation}
Now, reasoning in the same way as in the final part of the proof of Theorem \ref{L2-new}, we conclude that 
$$
\sup_{j\ge 0}\frac{s _{j+m}}{s _j}\le C\lambda^{m},\;\;m\in\mathbb{N},$$
whence $\log_2\lambda\ge{\delta^\infty}$. 

Next, since $\log_2\lambda<\delta$, for a fixed $n\in\mathbb{N}$ satisfying \eqref{eu23}, there exists $k\in\hj$ such that $s_{n+k}>\lambda^{n}s_k$ or $\nu_{n+k}<\nu_k,$ where as above $\nu_j=\lambda^{j}s_j^{-1}$, $j\in \hj.$ Consequently, in view of \eqref{equa25}, we have $\nu_k<\nu_{k-n}$. Then substituting $k-n$ for $k$ in \eqref{equa25},  we get $\nu_{k-n}<\nu_{k-2n}$ and so on. Thus, $(\nu _{k-rn})_{r=0}^\infty$ is an increasing  sequence.

If $j\le k$ and $m\in\mathbb{N},$ then we have
$$
k-r_1n\le j\le k-(r_1-1)n\;\;\mbox{and}\;\;k-(r_2+1)n\le j-m \le k-r_2n$$
for some $1\le r_1\le r_2$.
Hence, if $C_2:=\max_{l=1,2,\dots,n}\|\tau_{-l}\|_{E\to E}$, we get
\begin{equation*}
\frac{s_{j}}{s_{j-m}}\ge C_2^{-2}\frac{s_{k-r_1n}}{s_{k-r_2n}}=
C_2^{-2}\frac{\lambda^{-r_1n}\nu_{k-r_1n}^{-1}}{\lambda^{-r_2n}\nu_{k-r_2n}^{-1}}\ge C_2^{-2}\lambda^{n(r_2-r_1)}\;\;\mbox{for all}\;j\le k\;\mbox{and}\;m\in\mathbb{N},
\end{equation*}
and again as in the proof of Theorem \ref{L2-new}, for some $C>0$, it follows
$$
\sup_{j\ge 0}\frac{s _{-j-m}}{s
_{-j}}\le C\lambda^{-m},\;\;m\in\mathbb{N}.$$ 
Thus, $\log_2\lambda\le{\gamma^0}$.

As a result, implication \eqref{eu21} is proved, and it remains to show that the operator $\mn$ is not closed if $\lambda$ is equal to $2^\gamma,$ $2^{{\gamma^0}}$, $2^{{\delta^\infty}}$ or $2^\delta$. 

According to Theorem \ref{L2-new}, $\mn$ is an isomorphic mapping from $E$ onto $E$ and so it is a Fredholm operator with index $0$ for $\lambda\in (0,2^\gamma)\cup (2^\delta,\infty)$, while $\mn$ fails to be isomorphic if $\lambda$ is equal to $2^\gamma$ or $2^\delta$. Therefore, assuming that the operator $\mn$ is closed, say, for $\lambda=2^\gamma$, we would have that it is a Fredholm operator, whose index is $1$ or $-1$ (see the proof of Theorem \ref{L2-new}). But this contradicts the fact that the set of all Fredholm operators with a fixed index is open (see e.g. \cite[Theorem~III.22]{KG}). Moreover, if $\gamma<\delta^\infty<\gamma^0<\delta$, we know that $\mn$ is a Fredholm operator with index $1$ for $\lambda\in (2^{{\delta^\infty}},2^{{\gamma^0}})$ and it is not a closed operator (and so not a Fredholm operator) for $\lambda\in (2^\gamma,2^{{\delta^\infty}})\cup (2^{{\gamma^0}},2^\delta)$. Since the set of all Fredholm operators is open (see e.g. \cite[Theorem~III.21]{KG}), the operator $\mn$ is not closed if $\lambda$ is equal to $2^{{\gamma^0}}$ or $2^{{\delta^\infty}}$. The same reasoning may be applied also in the remaining cases. Thus, the proof is completed.
\end{proof}

\end{document}